\documentclass{amsart}
\usepackage{amsmath,amsthm,amscd,amssymb,latexsym,upref,stmaryrd}
\newcommand{\beq}{\begin{equation}}
\newcommand{\enq}{\end{equation}}

\newcommand{\bbR}{{\mathbb{R}}}

\newcommand{\cE}{{\mathcal E}}

\newcommand{\cL}{{\mathcal L}}

\newcommand{\no}{\nonumber}
\newcommand{\lb}{\label}
\newcommand{\f}{\frac}

\newcommand{\be}{{\mathbf e}}
\newcommand{\bu}{{\mathbf u}}

\newcommand{\bv}{{\mathbf v}}

\newcommand{\bx}{{\mathbf x}}

\numberwithin{equation}{section}

\renewcommand{\det}{\operatorname{det}}

\newcommand{\dom}{\operatorname{dom}}

\renewcommand{\Re}{\operatorname{Re }}

\newcommand{\diag}{\operatorname{diag}}

\newtheorem{theorem}{Theorem}[section]
\newtheorem{proposition}[theorem]{Proposition}
\newtheorem{lemma}[theorem]{Lemma}
\newtheorem{corollary}[theorem]{Corollary}
\newtheorem{definition}[theorem]{Definition}

\newtheorem{hypothesis}[theorem]{Hypothesis}

\theoremstyle{remark}
\newtheorem{remark}[theorem]{Remark}

\theoremstyle{definition}

\theoremstyle{remark}

\begin{document}

\title[Steadystates]{Stability of the steady states in  multidimensional reaction diffusion systems arising in combustion theory}


%    author one information
\author{Xinyao Yang}
\address{Department of Applied Mathematics, Xi'an Jiaotong - Liverpool University, Suzhou, China}
\email{xinyao.yang@xjtlu.edu.cn}

\author{Qingxia Li}
\address{Department of Mathematics and Computer Science,
	Fisk University, Nashiville, TN, USA}
\email{qli@fisk.edu}
%\thanks{}

%    author two information
%\author{}
%\address{}
%\curraddr{}
%\email{}
%\thanks{}

%    \subjclass is required.
\subjclass[2010]{35B35, 35K57}

\date{today}

%\dedicatory{}

%    Abstract is required.
\begin{abstract}
We prove that the steady state of a class of multidimensional reaction-diffusion systems is asymptotically stable at the intersection of unweighted space and exponentially weighted Sobolev spaces, and pay particular attention to a special case, namely, systems of equations that arise in combustion theory. The steady-state solutions considered here are the end states of the traveling fronts associated with the systems, and thus the present results complement recent papers \cite{GLS1, GLS2, GLS3, GLSR, GLY} that study the stability of traveling fronts.
\end{abstract}

\maketitle
\section{Introduction}\lb{sec1}
Traveling waves are a type of waves that maintain a certain shape while propagating at a fixed speed in a medium, and are widely present in a variety of natural phenomena modeled by nonlinear evolutionary equations that often describe chemical or physical processes shifting from one equilibrium state to another. In particular, the study of the stability of traveling waves is a burning topic in the natural sciences and engineering, the development of which involves many different methods and techniques, including weighted norms, graphical transformations, Lyapunov-Perron integrals of evolution equations, or corresponding sums of discrete systems. More extensive discussion of the use of different approaches to the study of such research can be found in \cite{BGM,BFZ,CL,Fife2002,Henry1981,kp,KV,liwu,pw94,Sa02,TZKS,vv02,Xin2000,rott3} and in their literature lists.

For the stability of travelling waves, we usually refer to an orbital stability in which any solution approaching the travelling wave is attracted to the solution itself or to one of its translations. In physical and mathematical sense, the instability of traveling waves can be understood as the sensitivity to perturbations, i.e., the possibility that the propagation of a traveling wave is distorted or altered from the system state due to a perturbation, eventually leading to abnormal appearance or abnormal steady-state output. A more precise definition can be found in \cite{Sa02}. B. Sandstede and A. Scheel \cite{SS01} have analysed various instability mechanisms in reaction-diffusion systems, their classification is based on the type of spectrum on the imaginary axis of the linear operator from linearization of the system. If only the isolated eigenvalues of the linear operator cross the imaginary axis, then the problem can be analyzed by the Lyapunov-Schmidt reduction, or by the theory of central manifolds. Whereas the essential instability \cite{SS01} arises when the essential spectrum (defined as consisting of all points on the spectrum that are not isolated eigenvalues of finite multiplicity \cite[Chapter 5]{Henry1981}) crosses the imaginary axis. The central manifold theorem or the Lyapunov-Schmidt method no longer applies in this case, and the proof of nonlinear stability is based on the use of exponential weights for the essential spectrum (see, e.g., \cite{pw94}), and on renormalization techniques to show that the nonlinear terms are asymptotically independent compared to linear diffusion.

The purpose of using an exponentially weighted space is that this allows one to shift the essential spectrum, which would otherwise cross the imaginary axis, to the left half of the complex plane, so that the exponentially decaying properties of the associated semigroup may be used. An application of this approach to the reaction diffusion equations can be found in a series of papers \cite{G,GLS1,GLS2,GLS3,GLSR} which demonstrate the orbital stability of the traveling front by studying perturbations that are small in both unweighted and weighted Sobolev spaces. Subsequently, in \cite{LSY}, the authors established the existence of a stable foliation near the traveling front solution of the reaction diffusion system in one-dimensional space, i.e., the existence of a central manifold at each point on the front solution that would attract nearby solutions that are slightly perturbed to the front solution itself or one of its translations, a result that complements the orbit stability results of \cite{GLS2}. Recently, we \cite{GLY} have extended the orbital stability result for the traveling front solution of the reaction-diffusion system associated with the combustion model in \cite{GLS2} to a multidimensional space. However, the result in \cite{GLY} is formulated under the assumption that the diffusion coefficients of the variables are the same throughout the system, although this assumption often does not satisfy the characteristics of the reaction-diffusion system as we perceive it in reality, for example, the diffusion coefficients of the reactants are usually smaller than the diffusion coefficients of the temperature. Therefore, this assumption renders our conclusions inapplicable for a significant number of reaction-diffusion systems. 

In this paper the objective is to study the stability of steady-state solutions to a class of reaction-diffusion systems associated with combustion problems in multidimensional space. We will first consider the following system of general reaction diffusion equations:
\begin{equation}\label{GEQ}
	\mathbf{u}_{t}(t, \mathbf{x})=D \Delta_{\mathbf{x}} \mathbf{u}(t, \mathbf{x})+f(\mathbf{u}(t, \mathbf{x})), \mathbf{u} \in \mathbb{R}^{n},  d \geq 2, t\geq 0, n\geq 2,
\end{equation}
where $D=\diag (d_1,d_2,... , d_n)$ and $\bx=(x_1,x_2,\ldots,x_d) \in \mathbb{R}^{d}$. We will later give some hypotheses about the nonlinear terms of this general system that allow us to apply the methods developed in \cite{GLS1,GLS2,GLS3}. A typical example is the following system:
\begin{equation}\label{MPr}
		\begin{cases}u_{1t}(t,\mathbf{x})=\Delta_\bx u_1(t,\bx)+u_2(t,\bx) g(u_1(t,\mathbf{x})),\, u_1,u_2\in\mathbb{R},\\
			u_{2t}(t,\bx)=\epsilon\Delta_\bx u_2(t,\bx)-\kappa u_2(t,\bx) g(u_1(t,\mathbf{x})), \,
			\bx\in\mathbb{R}^d,\end{cases}
\end{equation}
where
\begin{equation}\label{deng}
	g(u_1)=\begin{cases} e^{-\frac{1}{u_1}} &\mbox{ if } u_1>0 ;\\
		0 &\mbox{ if } u_1\leq 0,\end{cases}
\end{equation}
The parameters $\epsilon$ and $\kappa$ satisfy $0\leq \epsilon < 1$ and $\kappa > 0$. We assume $d\geq 2$ in this paper, but note that $d$ usually has a definite physical meaning only when $d=1,2,$ or $3$ is adopted.

The class of system \eqref{GEQ} admits various types of traveling wave solutions, such as wavefront, waveback and pulse (see, e.g., \cite{Jones,KSS,KV,LMNT}). We will consider a traveling wave solution moving in the direction of a given vector $\be \in \mathbb{R}^{d}$ with a constant speed $c>0$. Without loss of generality, we assume that $\be=(1,0, \ldots, 0)$. Consider a $t$-dependent change of variables $$z=x_{1}-c t, x_{j}=x_{j}, j=2, \ldots, d$$ in \eqref{GEQ}. Re-denoting $\bx=\left(z, x_{2}, \ldots, x_{d}\right)$ again, system \eqref{GEQ} in the new moving coordinates will become
\begin{equation}\label{eq1.1.2}
	\mathbf{u}_t(t,\mathbf{x})=D\Delta_{\mathbf{x}}\mathbf{u}(t,\mathbf{x})+c(\be\cdot\nabla_{\mathbf{x}})\mathbf{u}(t,\mathbf{x})+f(\mathbf{u}(t,\mathbf{x})).
\end{equation}
It can be shown that each solution of \eqref{GEQ} will correspond to a solution of \eqref{eq1.1.2}, and vice versa.

A traveling wave solution $\phi$ along $\be$ for \eqref{GEQ} is a $t$-independent solution $\phi=\phi(z)$ of \eqref{eq1.1.2}, that is, a function that depends only on $z$, the variable along $\be$, so that $\phi$ satisfies the ordinary differential equation
$$
D  \phi_{z z}(z)+c \phi_{z}(z)+f(\phi(z))=0.
$$
The traveling wave $\phi$ is called a planar front if there exist $\bx$-independent steady-state solutions of \eqref{eq1.1.2}, i.e., end states $\bu_-, \bu_+\in \bbR^n$, that satisfy the asymptotic relations
$$
\bu_{-}=\lim _{z \rightarrow-\infty} \phi(z), \bu_{+}=\lim _{z \rightarrow+\infty} \phi(z) .
$$
Such solutions are called \emph{pulses} if $\mathbf{u}_-=\mathbf{u}_+$ and \emph{fronts} if $\mathbf{u}_-\neq \mathbf{u}_+$. Under the circumstances of physical interest, these solutions approach both end states $\bu_-$ and $ \bu_+$ at an exponential rate, i.e., there exist numbers $K>0$ and $\omega_-<0<\omega_+$ such that $||\phi(z)-\mathbf{u}_-||\leq Ke^{-\omega_-z}$ for $z\leq0$ and  $||\phi(z)-\mathbf{u}_+||\leq Ke^{-\omega_+z}$ for $z\geq0$.

To study the stability of $\phi(z)$, we can perturb the function $\phi$ by either
\begin{itemize}
	\item[(i)] adding a function that depends only on one space variable $z$, that is, considering the solution $\mathbf{u}(t, \mathbf{x})$ of \eqref{eq1.1.2} with the initial condition
	\begin{equation*}
		\mathbf{u}(0, \mathbf{x})=\phi(\mathbf{x} \cdot \mathbf{e})+\mathbf{v}(0, \mathbf{x} \cdot \mathbf{e})
	\end{equation*}
	with some $\mathbf{v}: \mathbb{R} \times \mathbb{R} \rightarrow \mathbb{R}^{n}$ with some $\bv$ in the appropriate function space that we will construct later; or by
    \item[(ii)] adding a function that depends on all spatial variables, that is, considering the solution $\mathbf{u}(t, \mathbf{x})$ of \eqref{eq1.1.2} with the initial condition
	\begin{equation*}
		\mathbf{u}(0, \mathbf{x})=\phi(\mathbf{x} \cdot \mathbf{e})+\mathbf{v}(0, \mathbf{x})
	\end{equation*}
	with some $\mathbf{v}: \mathbb{R} \times \mathbb{R}^{d} \rightarrow \mathbb{R}^{n}$ from an appropriate function space.
\end{itemize}

Note that under the first type of perturbation, the problem is indeed very similar to an example in \cite[page 2440-2442]{GLS2}, except that the spatial variables will be multidimensional.  Thus, we focus here on the stability of the steady-state of the traveling front of system \eqref{eq1.1.2} under the second type of perturbation. The main progress of this paper, compared to our previous work, is the extension of the result \cite{GLS2}, for a system of one-dimensional spatial variables, to a system of multidimensional spatial variables, and the absence of that assumption in \cite{GLY} that the diffusion coefficients are the same for different system variables. Moreover, as we will describe in the text, this type of equations has a special "product-triangle" structure in the nonlinear reaction terms, which is similar to the equations studied in the one-dimensional case from \cite{GLS1,GLS2,GLS3}, and the class of nonlinear terms often appear in combustion models.

To better demonstrate how this "product-triangle" structure can help us study the stability of the steady states, we begin with the model case \eqref{MPr} for $\bu=$ $\left(u_{1}, u_{2}\right)^{T} \in \mathbb{R}^{2}$, in the moving coordinates $\bx=(z,x_2,...,x_d)$ the system will become
\begin{equation}\label{eq2.1.3}
	\bu_t(t,\bx)=\begin{pmatrix}
		1 & 0\\0 & \epsilon
	\end{pmatrix}\Delta_{\bx}\bu(t,\bx)+c(\be\cdot\nabla_\bx)\bu(t,\bx)+f(\bu(t,\bx)),
\end{equation}
where $f(\bu(t,\bx))=\begin{pmatrix}
	f_1(u_1,u_2)\\f_2(u_1,u_2)
\end{pmatrix}=\begin{pmatrix}
	u_2(t,\bx) g(u_1(t,\bx))\\-\kappa u_2(t,\bx)g(u_1(t,\bx))
\end{pmatrix}$. 

According to the discussion we presented before, a traveling front solution $\phi=\phi(z)$ is $t-$independent and will approach constant states $\bu_+$ and $\bu_-$ as $z\rightarrow\pm \infty$.  It is clear that the system \eqref{eq2.1.3} has two types of steady state solutions: one when $u_{1}(\mathbf{x})$ equals a real constant and $u_{2}(\mathbf{x})=0$, and the other when $u_{1}(\mathbf{x})=0$ and $u_{2}(\mathbf{x})$ is equal to a real constant. In particular, we can choose $u_{1}=1 / \kappa, u_{2}=0$, which is the equilibrium corresponding to the completely burned reactants, located behind the front, and $u_{1}=0, u_{2}=1$, corresponding to the unburned substances. In other words, we choose $\mathbf{u}_{-}=(1 / \kappa, 0)$ and $\mathbf{u}_{+}=(0,1)$. For more explanation why $\mathbf{u}_{-}$ and $\mathbf{u}_{+}$ are chosen this way, see \cite{GLS1} and Remark \ref{rem1}.

We will only discuss the stability of $\bu_-$ in detail here, since the stability of $\bu_+$ can be proved in precisely the same manner. We will investigate perturbations of the constant solution $\bu_-=(1/\kappa,0)$ that depend on all spatial variables of the system, that is, we consider the solutions $\bu(t,\bx)=\bu_-+\bv(t,\bx)$ of \eqref{eq2.1.3}  with the initial conditions
\begin{equation*}
	\bu(0,\bx)=\bu_-+\bv(0,\bx),
\end{equation*}
where $\bv=(v_1,v_2):\mathbb{R}_+\times\mathbb{R}^d\rightarrow\mathbb{R}^2$ is taken from an appropriate function space. Substituting $\bu(t,\bx)=\bu_-+\bv(t,\bx)$ into system \eqref{eq2.1.3}, we have:
\begin{equation}\label{eq2.1.4}
	\bv_t(t,x)=\begin{pmatrix}
		1 & 0\\0 & \epsilon
	\end{pmatrix} \Delta_{\bx}\bv(t,\bx)+c\partial_z\bv(t,\bx)+f(\bu_-+\bv(t,\bx)).
\end{equation}

Linearizing the nonlinearity $f(\bu_-+\bv(t,\bx))$ at $\bu_-=(1/\kappa,0)$ gives:
\begin{align*}
	f(\bu_-+\bv(t,\bx))&=f(\bu_-)+\partial_{\bu}f(\bu_-)\bv(t,\bx)+H(\bv(t,\bx))\\
	&=\begin{pmatrix}
		0 \\0
	\end{pmatrix}+\begin{pmatrix}
		0 & e^{-\kappa}\\
		0 &-\kappa e^{-\kappa}
	\end{pmatrix}\begin{pmatrix}
		v_1(t,\bx)\\ v_2(t,\bx)
	\end{pmatrix}+H(\bv(t,\bx)),
\end{align*}
where we introduced the nonlinear term by
\begin{equation}\label{eq2.1.6}
	H(\bv(t,\bx))=f(\bu_-+\bv(t,\bx))-f(\bu_-)-\partial_{\bu}f(\bu_-)\bv(t,\bx).
\end{equation}
We therefore have the following semilinear equation for the perturbations of the end state $\bu_-$,
\begin{equation}\label{eq2.1.7}
	\bv_t(t,\bx)=\begin{pmatrix}
		1 & 0 \\ 0 & \epsilon
	\end{pmatrix}\Delta_{\bx}\bv(t,\bx)+c\partial_z\bv(t,\bx)+\begin{pmatrix}
		0 & e^{-\kappa}\\0 &-\kappa e^{-\kappa}
	\end{pmatrix}\bv(t,\bx)+H(\bv(t,\bx)).
\end{equation}

We will show that the spectrum of the linear operator in \eqref{eq2.1.7} touches the imaginary axis in section \ref{sec2}, so that the weight function and the weighted functional space need to be introduced to further investigate the stability of the system under perturbations.

For $\mathcal{E}_{0}$ being the Sobolev spaces $H^{k}\left(\mathbb{R}^{d}\right)\left(k=1,2, \ldots\right.$ and we often define $H^{0}\left(\mathbb{R}^{d}\right)=$ $L^{2}\left(\mathbb{R}^{d}\right)$ ), which are suited for the study of nonlinear stability because they are closed under multiplication, we denote the norm in $\mathcal{E}_{0}$ by $\|\cdot\|_{0}$. Furthermore, we define the weight function of class $\alpha \in \mathbb{R}$ by 
$$
\gamma(\bx)=\gamma_{\alpha}\left(z, x_{2}, \ldots, x_{d}\right)=e^{\alpha} z, \text { for } \bx=\left(z, x_{2}, \ldots, x_{d}\right) \in \mathbb{R}^{d} .
$$
For a fixed weight function $\gamma_{\alpha}$ we define $\mathcal{E}_{\alpha}=\left\{u: \gamma_{\alpha} u \in \mathcal{E}_{0}\right\}$, with the norm $\|u\|_{\alpha}=\left\|\gamma_{\alpha} u\right\|_{0}$. Note that by this definition, $\mathcal{E}_{\alpha}=H_{\alpha}^{k}(\mathbb{R}) \otimes H^{k}\left(\mathbb{R}^{d-1}\right)$. Here and below we use the fact that $H^{k}\left(\mathbb{R}^{d}\right)$ can be written as the tensor product $H^{k}\left(\mathbb{R}^{d}\right)=H^{k}(\mathbb{R}) \otimes$ $H^{k}\left(\mathbb{R}^{d-1}\right)$. For general results on tensor products and operators on tensor products we refer to \cite[Section VIII.10]{RS1}. For ease of notation, we will use $H_{\alpha}^{k}\left(\mathbb{R}^{d}\right)=\left\{u: e^{\alpha z} u \in H^{k}\left(\mathbb{R}^{d}\right)\right\}$ to denote the weighted Sovolev space.

Although this weighted functional space solves the problem of spectral instability of the linear operator, it would also pose another difficulty in that the nonlinear terms cannot be controlled in the weighted space. Hence we introduce a new weighted space using the approach originally proposed by \cite{pw94} in the context of the Hamiltonian:
\begin{equation}\lb{spE}
	\mathcal{E}:=\mathcal{E}_{0} \cap \mathcal{E}_{\alpha}, \text { with }\|u\|_{\mathcal{E}}=\max \left\{\|u\|_{0},\|u\|_{\alpha}\right\} .
\end{equation}

We will prove the following theorem at the end of Section \ref{sec2}. Specifically, when considered in coordinates moving with fronts, we can show that the steady state of a nonlinear model problem of the form \eqref{eq2.1.3} is asymptotically stable in an orbital sense in a carefully chosen exponentially weighted space, i.e., the solution near the steady state  converges exponentially to the steady state solution itself in the weighted norm as long as the initial perturbation is sufficiently small in both the weighted and unweighted norm.

Finally in Section \ref{sec3}, we will summarize some key features of the system being used in the model problem \eqref{MPr} and generalize them into some hypotheses, thus for a general reaction-diffusion system \eqref{GEQ}, we can also prove the orbital stability of the steady state of a traveling front when it satisfies these hypotheses. Moreover, these hypotheses are often very common in reaction-diffusion systems associated with combustion problems.

\section{Stability of the steady states for the Model Case}\lb{sec2}
In this pilot section we study the stability of the end state of systems of the model problem \eqref{MPr}. This section is organized as follows. We study the spectrum of the operator generated by linearizing \eqref{eq2.1.3} about the end state in both unweighted and weighted spaces in subsection \ref{ssec1}. The Lipschitz property of the nonlinear term $H(\bv(t, \bx))$ is shown in subsection \ref{ssec2}, and the stability of the constant steady-state solution $\bu_{-}$ is proved in subections \ref{ssec3} and \ref{ssec4}.

\subsection{The setting in the model case}\lb{ssec1}

The information about the stability of the steady state of system \eqref{eq2.1.3} is often disclosed by the information about the spectrum of the linear operator obtained by linearizing \eqref{eq2.1.3} about the steady state. Therefore, we first define the linear differential expression in \eqref{eq2.1.7} by $L$,
\begin{equation}\label{eq2.1.8}
	L=\begin{pmatrix}
		1 & 0 \\ 0 &\epsilon
	\end{pmatrix}
	\Delta_{\bx}+c\partial_z+\begin{pmatrix}
		0 & e^{-\kappa}\\ 0&-\kappa e^{-\kappa}
	\end{pmatrix}.
\end{equation}
We will consider a differential operator $\mathcal{L}$ associated with the differential expression $L$ in the Sobolev space $H^k(\mathbb{R}^d)^2$ of vector-valued functions and throughout assume that $k\geq [\dfrac{d+1}{2}]$. 
\begin{definition}
	We call a time-independent solution spectrally stable in a space $\mathcal{X}$ if the spectrum $\operatorname{Sp}(\mathcal{L})$ of the operator $\mathcal{L}$ is contained in $\{\lambda: \operatorname{Re} \lambda<a\} \cup\{0\}$ for some $a>0$.
\end{definition}

As we will see soon, the essential spectrum of $\mathcal{L}$ touches the imaginary axis. This prevents $\bu_-$ from being stable in the space $H^k(\mathbb{R}^d)^2$, and we will have to replace the space by the weighted space $H^k_{\alpha}(\mathbb{R}^d)^2$, with an exponential weight with respect to the variable $z$. But then in this new space the nonlinearity will lose the local Lipschitz property needed to see the well-poseness of \eqref{eq2.1.7}. To gain it back, as in \cite{GLS1,GLS2,GLS3,GLY}, we will pass to the intersection space $H^k(\mathbb{R}^d)^2\cap H^k_{\alpha}(\mathbb{R}^d)^2$, and perform further analysis there.

\begin{remark}\label{remark}
	We first need information about the spectra of the linear operators associated with \eqref{eq2.1.8}, so several operators in the different spaces considered below will be involved. We will use the following notation for these operators. If $B$ is a general $(2\times 2)$ system of $n$ differential expressions as, for instance, in \eqref{eq2.1.8}, then we shall use notation $\mathcal{B}:\mathcal{E}_0^2\rightarrow \mathcal{E}_0^2$ and $\mathcal{B}_{\alpha}:\mathcal{E}_{\alpha}^2\rightarrow \mathcal{E}_{\alpha}^2$ to denote the linear operator in $\mathcal{E}_{0}^2$ and $\mathcal{E}_{\alpha}^2$, respectively, given by the formula $u\rightarrow Bu$, with their \textquotedblleft natural\textquotedblright domains, that is, for $k=0,1,\cdots$, we use $\mathcal{L}:\mathcal{E}_0^2\rightarrow\mathcal{E}_0^2$ to denote the linear operator given by the formula  $u\mapsto L u$, whose domain is $H^{k+2}(\mathbb{R}^d)^2$; and use $\mathcal{L}_{\alpha}:\mathcal{E}_{\alpha}^2\rightarrow\mathcal{E}_{\alpha}^2$ to denote the operator in $\mathcal{E}_{\alpha}^2$ given by the formula  $u\mapsto L u$, whose domain is the set of $(u_1,u_2)$ where  $\gamma_{\alpha}u_1,\gamma_{\alpha}u_2\in H^{k+2}(\mathbb{R}^{d})$. We use notation $\mathcal{L}_{\mathcal{E}}:\mathcal{E}^2\rightarrow\mathcal{E}^2$ to denote the linear operator given by $u\rightarrow L u$ with the domain of $\mathcal{L}_{\mathcal{E}}$ being the set of $(u_1, u_2)$ satisfying $(u_1,u_2)\in\dom(\mathcal{L})\cap\dom(\mathcal{L}_{\alpha})$, where $\dom(\mathcal{L})$ and $\dom(\mathcal{L}_{\alpha})$ are respective domains defined above.
	\hfill$\Diamond$	
\end{remark}

First, we will use Fourier transform to explore the spectrum of the constant coefficient differential operator $\mathcal{L}$ on $L^2(\mathbb{R}^d)$, and the spectrum of the constant coefficient differential operator $\mathcal{L}_{\alpha}$ on $L^2_{\alpha}(\mathbb{R})\otimes L^2(\mathbb{R}^{d-1})$, respectively. We will use the following elementary proposition to show that the spectrum of $\mathcal{L}$ on $\mathcal{E}_0^2$ touches the imaginary axis, and the spectrum of $\mathcal{L}_{\alpha}$ on $\mathcal{E}_{\alpha}^2$ will be away from the imaginary axis. 

\begin{proposition}\label{pro2.2.3}
	Assume that $\mathcal{L}$ and $\mathcal{L}_{\alpha}$ are the constant coefficient linear differential operators associated with the differential expression $L$ in \eqref{eq2.1.8}. On the unweighted space $\mathcal{E}_0^2=H^k(\mathbb{R}^d)^2$ for all integers $k\geq 0$, one has
	\begin{equation*}
		\sup\{ \Re\lambda:\lambda\in \operatorname{Sp}(\mathcal{L}) \}=0,
	\end{equation*} 
	so that the spectrum of $\mathcal{L}$ touches the imaginary axis. By choosing $\alpha\in(0,c/2)$, one has $$\sup\{ \Re\lambda: \lambda\in \operatorname{Sp}(\mathcal{L}_{\alpha}) \}<-\nu$$ for some $\nu>0$ so that the spectrum of $\mathcal{L}_{\alpha}$ is shifted to the left of the imaginary axis on the weighted space $\mathcal{E}_{\alpha}^2=H^k_{\alpha}(\mathbb{R})^2\otimes H^k(\mathbb{R}^{d-1})^2$.
	
	Furthermore, there exists $K>0$ such that $\left\|e^{t \mathcal{L}_{\alpha}}\right\|_{\cE_{\alpha}^{2}\rightarrow \cE_{\alpha}^{2}} \leqslant K e^{-\nu t}$ for $t \geqslant 0$
\end{proposition}

\begin{proof}
	By Lemma \ref{lem2.2.1} proved next, it is enough to consider the case $k=0$, that is, to assume that $\cE_0=L^2(\bbR^d)$. To find $\operatorname{Sp}(\mathcal{L})$ in the unweighted space $\mathcal{E}_0^2$, we can use Fourier transform. By properties of Fourier transform, see, e.g., \cite[Section 6.5]{EngelNagel}, the operator $\mathcal{L}$ on $L^2(\bbR^d)^2$ is similar to the operator on $L^2(\bbR^d)^2$ of multiplication by the matix-valued function 
	\begin{equation}\label{Mxi}
		M(\xi)=-(\xi_1^2+\xi_2^2+\cdots+\xi_d^2)\begin{pmatrix}
			1 & 0 \\ 0 &\epsilon
		\end{pmatrix}+i\xi_1 c\begin{pmatrix}
			1 &0 \\ 0&1
		\end{pmatrix}+\begin{pmatrix}
			0 & e^{-\kappa} \\ 0 &-\kappa e^{-\kappa}
		\end{pmatrix},
	\end{equation}
	where $\xi=(\xi_1,...,\xi_d)\in\mathbb{R}^d$.
	Thus the spectrum of $\mathcal{L}$ on $L^2(\bbR^d)^2$ is the closure of the union over $\xi\in\mathbb{R}^d$ of the spectra of the matrices $M(\xi)$. Hence the spectrum of $\mathcal{L}$ is equal to the closure of the set of $\lambda\in\mathbb{C}$ for which there exists $\xi\in\mathbb{R}^d$ such that 
	\begin{align*}
		\det \big(M(\xi)-\lambda I\big)&=\det\left(-(\xi_1^2+\xi_2^2+\cdots+\xi_d^2)\left(\begin{smallmatrix}
			1 & 0 \\ 0 &\epsilon
		\end{smallmatrix}\right)+i\xi_1 c I+\left(\begin{smallmatrix}
			0 & e^{-\kappa} \\ 0 &-\kappa e^{-\kappa}
		\end{smallmatrix}\right)\right)\\
		&=0.
	\end{align*} 
	It is a collection of curves $\lambda=\lambda(\xi)$, where $\lambda(\xi)$ are the eigenvalues of the matrices $M(\xi)$. Thus the spectrum of the operator $\mathcal{L}$ is 
	\begin{align}\label{spectrum1}
		\operatorname{Sp}(\mathcal{L})&=\displaystyle \mathop{\cup}_{\xi\in\mathbb{R}^d}\operatorname{Sp}\begin{pmatrix}
			-(\xi_1^2+\cdots+\xi_d^2)+ci\xi_1 & e^{-\kappa} \\ 0 & -\epsilon (\xi_1^2+\cdots+\xi_d^2)+ci\xi_1-\kappa e^{-\kappa}\end{pmatrix}\\\no
		&=\displaystyle\mathop{\cup}_{\xi\in\mathbb{R}^d}(-(\xi_1^2+\cdots+\xi_d^2)+ci\xi_1)\bigcup\displaystyle\mathop{\cup}_{\xi\in\mathbb{R}^d}(-\epsilon (\xi_1^2+\cdots+\xi_d^2)+ci\xi_1-\kappa e^{-\kappa}).
	\end{align}
	
	This implies that the spectrum of $\mathcal{L}$ in $L^2(\bbR^d)^2$ touches the imaginary axis when $\xi=(\xi_1,...,\xi_d)=(0,...,0)$.
	
	We also need $\operatorname{Sp}(\mathcal{L}_{\alpha})$ on the weighted space $\mathcal{E}_{\alpha}^2$. First define the linear map $N:\mathcal{E}_{\alpha}\mapsto\mathcal{E}_0$ given by $N v=\gamma_{\alpha}v$, and notice that by definition $N$ is an isomorphism of $\mathcal{E}_{\alpha}$ onto $\mathcal{E}_0$. In particular, we can define a linear operator $\hat{\mathcal{L}}=N \mathcal{L}_{\alpha} N^{-1}$ on $\mathcal{E}_0^2=L^2(\mathbb{R}^d)^2$, with the domain $\dom(\hat{\mathcal{L}})=H^{2}(\mathbb{R}^d)^2$ since $N^{-1}$ maps $\dom(\hat{\mathcal{L}})$ in $\dom(\mathcal{L}_{\alpha})$. The operator $\hat{\mathcal{L}}$ is similar to $\mathcal{L}_{\alpha}$ on $\mathcal{E}_{\alpha}^2$ and hence has the same spectrum. 
	
	In particular, let us consider the operator $\partial_{z,\alpha}$ on $\mathcal{E}_{\alpha}$ with $$\dom(\partial_{z,\alpha})=H^{1}_{\alpha}(\mathbb{R})\otimes H^{1}(\mathbb{R}^{d-1}).$$
	Fix any $v\in H^{1}(\mathbb{R}^{d})=\dom(\hat{\partial_z})$ when $\hat{\partial_z}$ is considered in $L^2(\mathbb{R}^d)$, and $\hat{\partial_z}=N\partial_{z,\alpha}N^{-1}$. Then, temporarily redenoting $\gamma_{\alpha}(z)=e^{\alpha z}$, we have
	\begin{align*}
		\partial v=N\partial_{z,\alpha}N^{-1}v&=\gamma_{\alpha}\partial_z(\gamma_{-\alpha}v)=\gamma_{\alpha}(\gamma_{-\alpha}'v+\gamma_{-\alpha}\partial_zv)\\
		&=\gamma_{\alpha}(-\alpha\gamma_{-\alpha}v+\gamma_{-\alpha}\partial_zv)\\
		&=(\partial_z-\alpha) v.
	\end{align*}
	
	Denoting $y=(x_2,...,x_d)$, then $\bx=(z,y)\in\mathbb{R}^d$, a similar computation shows that for each 
	\begin{equation*}
		\bv=(v_1,v_2)^T\in\dom\hat{\mathcal{L}}=H^{2}(\mathbb{R}^d)^2\subset L^2(\mathbb{R}^d)^2,
	\end{equation*}
	we have:
	\begin{align*}
		\hat{\mathcal{L}}\bv=& e^{\alpha z}\left(\begin{pmatrix}
			1 & 0\\0&\epsilon
		\end{pmatrix}(\Delta_y+\partial_{zz})+c\partial_z+\begin{pmatrix}
			0 & e^{-\kappa}\\0&-\kappa e^{-\kappa}
		\end{pmatrix}\right) ( e^{-\alpha z}\bv)\\
		=&\left(\begin{smallmatrix}
			1 & 0\\0&\epsilon
		\end{smallmatrix}\right)\Delta_{y} \bv+\left(\begin{smallmatrix}
			1 & 0\\0&\epsilon
		\end{smallmatrix}\right)( \alpha^2\bv-2\alpha \partial_z\bv+\partial_{zz}\bv )+c(\partial_z\bv-\alpha \bv)+\left(\begin{smallmatrix}
			0 & e^{-\kappa}\\0&-\kappa e^{-\kappa}
		\end{smallmatrix}\right)\bv
		\\
		=&\left(\begin{smallmatrix}
			1 & 0\\0&\epsilon
		\end{smallmatrix}\right)\Delta_{\bx} \bv+ \left(cI-2\alpha\left(\begin{smallmatrix}
			1 & 0\\0&\epsilon
		\end{smallmatrix}\right)\right)\partial_z\bv+\left(\alpha^2\left(\begin{smallmatrix}
			1 & 0\\0&\epsilon
		\end{smallmatrix}\right)-c\alpha I+\left(\begin{smallmatrix}
			0 & e^{-\kappa}\\0&-\kappa e^{-\kappa}
		\end{smallmatrix}\right)\right)\bv.
	\end{align*}
	
	Via the Fourier transform, the operator $\hat{\mathcal{L}}$ on $L^2(\bbR^d)^2$ is similar to the operator of multiplication on $L^2(\bbR^d)^2$ by the matrix-valued function 
	\begin{align*}
		N(\xi)& =-\|\xi\|^2\left(\begin{smallmatrix}
			1&0\\0&\epsilon
		\end{smallmatrix}\right)+(i\xi_1c-\alpha c)I+(\alpha^2-2i\xi_1\alpha)\left(\begin{smallmatrix}
			1&0\\0&\epsilon
		\end{smallmatrix}\right)+\left(\begin{smallmatrix}
			0 & e^{-\kappa}\\0&-\kappa e^{-\kappa}
		\end{smallmatrix}\right)\\
		&=\left(\begin{smallmatrix}
			-\|\xi\|^2+(c-2\alpha)i\xi_1+\alpha^2-\alpha c& e^{-\kappa}\\
			0 & -\epsilon \|\xi\|^2+(c-2\alpha\epsilon)i\xi_1+\alpha^2\epsilon-c\alpha-\kappa e^{-\kappa}
		\end{smallmatrix}\right)
	\end{align*}
	where $\|\xi\|^2=\xi_1^2+\cdots+\xi_d^2$.
	Hence,
	\begin{align}\label{spectrum2}
	\operatorname{Sp}(\mathcal{L}_{\alpha})
		&=\displaystyle\mathop{\cup}_{\xi\in\mathbb{R}^d}(-(\xi_1^2+\cdots+\xi_d^2)+(c-2\alpha)i\xi_1+\alpha^2-c\alpha)\no\\
		&\qquad\bigcup\displaystyle\mathop{\cup}_{\xi\in\mathbb{R}^d}(-\epsilon (\xi_1^2+\cdots+\xi_d^2)+(c-2\alpha\epsilon)i\xi_1+\alpha^2\epsilon-c\alpha-\kappa e^{-\kappa}).
	\end{align}
	Then
	\begin{align*}
		\sup\{ \Re\lambda:\lambda\in\operatorname{Sp}(\mathcal{L}_{\alpha}) \}&=\sup \{ \Re\lambda:\lambda\in \operatorname{Sp}(\hat{\mathcal{L}}) \}\\
		&=\max(\alpha^2-c\alpha,\epsilon\alpha^2-c\alpha-\kappa e^{-\kappa})\\
		&=\alpha^2-c\alpha.
	\end{align*}
	Thus we conclude that for $\alpha\in(0,c/2)$, one has $\sup\{ \Re\lambda:\lambda\in \operatorname{Sp}(\mathcal{L}_{\alpha}) \}<0$ so that the spectrum $\operatorname{Sp}(\mathcal{L}_{\alpha})$ on the weighted space $\mathcal{E}_{\alpha}^2$ is moved to the left of the imaginary axis.
	
	Furthermore, the operator $\mathcal{L}_{\alpha}$ associated with the differential expression $L$ in \eqref{eq2.1.8} generates an analytic semigroup provided $\epsilon>0$ and a strongly continuous semigroup provided $\epsilon=0$. As shown in \cite{GLS1}, in either case $\mathcal{L}$ enjoys the spectral mapping property, that is, the boundary of the spectrum of the semigroup operator $e^{t \mathcal{L}_{\alpha}}$ is controlled by the boundary of the spectrum of the semigroup generator $\mathcal{L}_{\alpha}$ for any $\epsilon \geqslant 0$. Then by the above mentioned semigroup property, see, e.g. \cite[Proposition 4.3]{GLS1}, there exists $K>0$ such that $\left\|e^{t \mathcal{L}_{\alpha}}\right\|_{\cE_{\alpha}^{2} \rightarrow \cE_{\alpha}^2} \leqslant K e^{-\nu t} .$
\end{proof}

\begin{lemma}\label{lem2.2.1}
	The linear constant coefficient differential operator $\mathcal{L}$ associated with $L$ defined in \eqref{eq2.1.8} has the same spectrum on $L^2(\mathbb{R}^d)^2$ and on $H^k(\mathbb{R}^d)^2$ for all integers $k> 0$; similarly, the operator $\mathcal{L}_{\alpha}$ associated with $L$ defined in \eqref{eq2.1.8} has the same spectrum on $L^2_{\alpha}(\mathbb{R})\otimes L^2(\mathbb{R}^{d-1})$ and on $H^k_{\alpha}(\mathbb{R})\otimes H^k(\mathbb{R}^{d-1})$, for all integers $k> 0$. 
\end{lemma}
\begin{proof}
	To show that the spectrum of $\mathcal{L}$ in $H^k(\mathbb{R}^d)^2$ is the same as the spectrum of $L^2(\mathbb{R}^d)^2$,
	we let $\mathcal{F}_1$ denote the Fourier transform acting from $H^k(\mathbb{R}^d)^2$ into $L^2_m(\mathbb{R}^d)^2$, where $L^2_m(\mathbb{R}^d)^2$ is the weighted $L^2$-space with the standard weight $m(\xi)=(1+|\xi|_{\mathbb{R}^d}^2)^{k/2}$. By the standard property of the Fourier transform we have $\mathcal{F}_1\Delta_x=-|\xi|_{\mathbb{R}^d}^2 \mathcal{F}_1$ and $\mathcal{F}_1 \partial_z=-i\xi_1\mathcal{F}_1$. Thus $\mathcal{F}_1\mathcal{L}=M\mathcal{F}_1$ for a matrix-valued function $M=M(\xi)$ obtained from \eqref{eq2.1.8} by replacing $\Delta_x$ by $-|\xi|_{\mathbb{R}^d}^2$ and $\partial_z$ by $-i\xi_1$.
	
	On the other hand, the operator of multiplication by $m(\cdot)$ is an isomorphism of $L^2_m(\mathbb{R}^d)^2$ onto $L^2(\mathbb{R}^d)^2$. Let us denote by $\mathcal{L}_{H^k}$ the operator $\mathcal{L}$ associated with $L$ on the space $H^k(\mathbb{R}^d)^2$, and by $\mathcal{L}_{L^2}$ the operator $\mathcal{L}$ associated with $L$ on the space $L^2(\mathbb{R}^d)^2$. By the previous paragraph we then have $m\mathcal{F}_1\mathcal{L}_{H^k}=Mm\mathcal{F}_1$. (Here and below we allow a slight abbreviation of notation, the proper writing is that $u\in \dom(\mathcal{L}_{H^k})$ implies $m\mathcal{F}_1u\in\dom(M)$ and $m\mathcal{F}_1\mathcal{L}_{H^k}u=Mm\mathcal{F}_1u$ for all $u\in\dom(\mathcal{L}_{H^k})$.)
	
	We remark that the operator of multiplication by $-i\xi_j$, $j=1,...,d$ on $L^2(\mathbb{R}^d)$ is similar to the operator of differentiation $\partial_{x_j}$ on $L^2(\mathbb{R}^d)$ via the Fourier transform $\mathcal{F}_2$. This implies that $\mathcal{F}_2\mathcal{L}_{L^2}=M\mathcal{F}_2$ with the same matrix-valued function $M$ as above. It follows that 
	\begin{equation}\lb{OPARES}
		\mathcal{L}_{H^k}=(m\mathcal{F}_1)^{-1}Mm\mathcal{F}_1=(m\mathcal{F}_1)^{-1}(\mathcal{F}_2\mathcal{L}_{L^2}\mathcal{F}_2^{-1})(m\mathcal{F}_1),
	\end{equation}
	therefore the spectrum of $\mathcal{L}$ on $H^k(\mathbb{R}^d)^2$ is the same as the spectrum of $\mathcal{L}$ on $L^2(\mathbb{R}^d)^2$ because the operators on $H^k(\mathbb{R}^d)^2$ and $L^2(\mathbb{R}^d)^2$ are similar.
	
	By analogous argument, the spectrum of $\mathcal{L}_{\alpha}$ on $L^2_{\alpha}(\mathbb{R})\otimes L^2(\mathbb{R}^{d-1})$ is the same as the spectrum of $\mathcal{L}_{\alpha}$ on $H^k_{\alpha}(\mathbb{R})\otimes H^k(\mathbb{R}^{d-1})$.
\end{proof}

\begin{remark}\label{spectp}
	Recall that we denote $y=(x_2,...,x_d)$. Let $\Delta_y$ be the operator given by the differential expression $\partial_{x_2}^2+\cdots+\partial_{x_d}^2$, where the domain of $\Delta_y$ on $H^k(\mathbb{R}^{d-1})$ is the set of $u$ such that $u\in H^{k+2}(\mathbb{R}^{d-1})$. We denote by $\mathcal{L}_{1,\alpha}:H^k_{\alpha}(\mathbb{R})^2\rightarrow H^k_{\alpha}(\mathbb{R})^2$ the operator given by the differential expression $\partial_{zz}+c\partial_z+\begin{pmatrix}
		0 & e^{-\kappa}\\0&-\kappa e^{-\kappa}
	\end{pmatrix}$, and $\dom(\mathcal{L}_{1,\alpha})=H^{k+2}_{\alpha}(\mathbb{R})^2\subset H^k_{\alpha}(\mathbb{R})^2$. The operator $\mathcal{L}_{\alpha}$ on $\mathcal{E}_{\alpha}^2=H^k_{\alpha}(\mathbb{R})^2\otimes H^k(\mathbb{R}^{d-1})^2$ can be written as $\mathcal{L}_{1,\alpha}\otimes I_{H^k(\mathbb{R}^{d-1})}+I_{H^k_{\alpha}(\mathbb{R})}\otimes \Delta_y$. We have yet another approach to prove Propositon \ref{pro2.2.3} by using \cite[Theorem XIII.34, Theorem XIII.35 and Corollary 1]{ReedSimon4} . Indeed, since $\mathcal{L}_{1,\alpha}$ and $\Delta_y$ are the generators of bounded semigroups on $H^k_{\alpha}(\mathbb{R})$ and $H^k(\mathbb{R}^{d-1})$ respectively, we have (see \cite{ReedSimon4})
	\begin{equation}\label{tensorform}
		\operatorname{Sp}(\mathcal{L}_{1,\alpha}\otimes I_{H^k(\mathbb{R}^{d-1})}+I_{H^k_{\alpha}(\mathbb{R})}\otimes
		\Delta_y)=\operatorname{Sp}(\mathcal{L}_{1,\alpha})+\operatorname{Sp}(\Delta_y).
	\end{equation}
	 Thus, $\operatorname{Sp}(\cL_{\alpha})=\operatorname{Sp}(\cL_{1,\alpha})+\operatorname{Sp}(\Delta_y)$. It is easy to see that the spectrum of $\Delta_y$ on $H^k(\mathbb{R}^{d-1})$ is the non-negative semiline $(-\infty, 0]$ and we have showed that the spectrum of $\mathcal{L}_{1,\alpha}$ on $H^k_{\alpha}(\mathbb{R})$ satisfies $\sup\{ \Re\lambda: \lambda\in \operatorname{Sp}(\mathcal{L}_{1,\alpha}) \}<-\nu$ for some $\nu>0$, thus Proposition \ref{pro2.2.3} is proved. Moreover, the same argument shows that if $\Gamma$ is the curve that bounds the spectrum of $\mathcal{L}_{1,\alpha}$ on the right, then $\operatorname{Sp}(\mathcal{L}_{\alpha})$ is the entire solid part of the plane bounded by $\Gamma$.
\end{remark}

Now notice that the differential expression $L$ in \eqref{eq2.1.8} has the following triangular structure,
\begin{equation}\lb{trs}
	L=\begin{pmatrix}
		\Delta_{\bx}+c\partial_{z} & e^{-\kappa}\\
		0 & \epsilon\Delta_{\bx}+c\partial_{z}-\kappa e^{-\kappa}
	\end{pmatrix}.
\end{equation}

Let 
\begin{align}\label{eq2.2.5}
	&L^{(1)}=\Delta_{\bx}+c\partial_{z};\\
	\label{eq2.2.6}
	&L^{(2)}=\epsilon\Delta_{\bx}+c\partial_{z}-\kappa e^{-\kappa},
\end{align}
and for $i=1,2$, let $\mathcal{L}^{(i)}$ be the operator on $H^k(\mathbb{R}^d)$ defined by $v_i\mapsto L^{(i)}v_i$, with the domain of $\mathcal{L}^{(i)}$ to be $H^{k+2}(\mathbb{R}^d)$, for $k=0,1,2,...$.
\begin{lemma}\label{lem2.2.4}Consider the operators $\mathcal{L}^{(1)}$ and $\mathcal{L}^{(2)}$ on $H^k(\mathbb{R}^d)$ defined by the differential expressions $L^{(1)}$ and $L^{(2)}$ given in \eqref{eq2.2.5} and \eqref{eq2.2.6}.
	\begin{itemize}
		\item[(1)] The operator $\mathcal{L}^{(1)}$ generates a bounded strongly continuous semigroup on\\
		$H^k(\mathbb{R}^d)$;
		\item[(2)] The operator $\mathcal{L}^{(2)}$ satisfies $\sup\{\Re\lambda:\lambda\in \operatorname{Sp}(\mathcal{L}^{(2)})\}<0$ on $H^k(\mathbb{R}^d)$;
		\item[(3)] The following is true on $H^k(\mathbb{R}^d)$:
		\begin{itemize}
			\item[(a)] $\sup\{\Re\lambda:\lambda\in \operatorname{Sp}(\mathcal{L}^{(1)})\}\leq 0$;
			\item[(b)] There exist $K>0$ and $\rho>0$ such that for the strongly continuous semigroup $\{ e^{t\mathcal{L}^{(2)}} \}_{t\geq 0}$, one has $\|e^{t\mathcal{L}^{(2)}}\|_{H^k(\bbR^d)\rightarrow H^k(\bbR^d)}\leq Ke^{-\rho t}$ for all $t\geq 0$.
		\end{itemize}
	\end{itemize} 
\end{lemma}

\begin{proof}
	As in Lemma \ref{lem2.2.1}, we can prove that the operators $\mathcal{L}^{(i)}$, $i=1,2$ have the same spectrum on $H^k(\mathbb{R}^d)$ and on $L^2(\mathbb{R}^d)$. 
	
	Using the Fourier transform, we find that the spectrum of $\mathcal{L}^{(1)}$ on $L^2(\mathbb{R}^d)$ is the union of the curves $\lambda_1(\xi)=-(\xi_1^2+\cdots+\xi_d^2)+ci\xi_1$ for all $\xi=(\xi_1,...,\xi_d)\in\mathbb{R}^d$,  thus $\sup\{\Re\lambda:\lambda\in \operatorname{Sp}(\mathcal{L}^{(1)})\}\leq 0$ on $L^2(\mathbb{R})$ which proves (3)(a). By the proof of Proposition A.1(1) in \cite{GLS2}, the operator $\mathcal{L}^{(1)}$ generates a bounded semigroup on $L^2(\mathbb{R}^d)$. Operators on $H^k(\mathbb{R}^d)$ and $L^2(\mathbb{R}^d)$ associated with the same constant-coefficient differential expression are similar, see \eqref{OPARES}, therefore the semigroup they generate are similar, so (1) is proved. 
	
	The spectrum of $\mathcal{L}^{(2)}$ on $L^2(\mathbb{R}^d)$ is the union of the curves $\lambda_2(\xi)=-\epsilon(\xi_1^2+\cdots+\xi_d^2)+ci\xi_1-\kappa e^{-\kappa}$ for all $\xi=(\xi_1,...,\xi_d)\in\mathbb{R}^d$, and therefore $\sup\{\Re\lambda:\lambda\in \operatorname{Sp}(\mathcal{L}^{(2)})\}< 0$ on $L^2(\mathbb{R}^d)$, also on $H^k(\mathbb{R}^d)$ by Lemma \ref{lem2.2.1}, proving (2).
	
	Assertion (3)(b) is a direct consequence of (2), see \cite{GLS2} Lemma 3.13.
\end{proof}

\begin{remark}\label{rem1}
	To conclude this subection we explain why the end states $\mathbf{u}_-$ and $\mathbf{u}_+$ for the model system
	\begin{equation}
		\begin{cases}u_{1t}(t,\mathbf{x})=\partial_{zz} u_1(t,\bx)+c\partial_z u_1+ u_2(t,\bx) g(u_1(t,\mathbf{x})),\, u_1,u_2\in\mathbb{R},\\
			u_{2t}(t,\bx)=\epsilon\partial_{zz} u_2(t,\bx)+c\partial_z u_2-\kappa u_2(t,\bx) g(u_1(t,\mathbf{x})), \,
			\bx\in\mathbb{R}^d,\end{cases}
	\end{equation}
	where $g$ is defined in \eqref{deng}, were chosen as $\bu_-=(1/\kappa,0)$ and $\bu_+=(0,1)$.
	
	Let $\Phi=(\phi_1,\phi_2)$ be a time-independent solution of the model system so that $\Phi$ satisfies the ODE system
	\begin{equation}\label{eq1.7.2}
		\begin{cases}\partial_{zz} \phi_1( \bx)+c\partial_z \phi_1+\phi_2( \bx) g(\phi_1( \mathbf{x}))=0,\\
			\epsilon\partial_{zz} \phi_2( \bx)+c\partial_z \phi_2-\kappa \phi_2( \bx) g(\phi_1( \mathbf{x}))=0.\end{cases}
	\end{equation} 
	We are interested in solutions of \eqref{eq1.7.2} that satisfy the boundary conditions at $z\rightarrow\pm\infty$,
	\[ 
	(\phi_1,\phi_2)(-\infty)=(\phi_1^{\star},0) , \, (\phi_1,\phi_2)(\infty)=(0,1).
	\]
	Such solutions represent traveling combustion fronts. Here, the left temperature $\phi_1^{\star}$ is an unknown to be determined.
	
	In the ODE system \eqref{eq1.7.2}, we set $\phi_3=\partial_z\phi_1$ and $\phi_4=\partial_z \phi_2$, and also use prime to denote the derivative with respect to $z$, to obtain the following first-order system:
	\begin{equation}\label{eq1.7.3}
		\phi_1'=\phi_3,
	\end{equation} 
	\begin{equation}
		\phi_2'=\phi_4,
	\end{equation}
	\begin{equation}\label{eq1.7.5}
		\phi_3'=-(c\phi_3+\phi_2g(\phi_1)),
	\end{equation}
	\begin{equation}\label{eq1.7.6}
		\phi_4'=-\frac{1}{\epsilon}[c\phi_4-\kappa \phi_2 g(\phi_1)].
	\end{equation}
	By adding \eqref{eq1.7.5} to \eqref{eq1.7.6} multiplied by $\epsilon/\kappa$, we obtain the following equation,
	\begin{equation}
		\phi_1''+c\phi_1'+\frac{\epsilon}{\kappa}\phi_2''+\frac{c}{\kappa}\phi_2'=0.
	\end{equation}
	This expression can be integrated once to produce a function of $z$ that is constant along any traveling wave. We denote this constant by $k$ so that
	\begin{equation}\label{eq1.7.8}
		\phi_3+c\phi_1+\frac{\epsilon}{\kappa}\phi_4+\frac{c}{\kappa}\phi_2=constant:=k.
	\end{equation}
	For the solution that approaches $(\phi_1,\phi_2,\phi_3,\phi_4)=(0,1,0,0)$ as $z\rightarrow\infty$, we must have $k=\frac{c}{\kappa}.$ Substituting $k=\frac{c}{\kappa}$ into equation \eqref{eq1.7.8} we have
	\begin{equation}
		\phi_3=-c\phi_1-\frac{\epsilon}{\kappa}\phi_4-\frac{c}{\kappa}\phi_2+\frac{c}{\kappa}\rightarrow 0,\  \phi_1\rightarrow\phi_1^*,\  \phi_2\rightarrow 0 \ \text{and}\ \phi_4\rightarrow 0\ \text{as }\ z\rightarrow -\infty
	\end{equation}
	since the steady solution of the system \eqref{eq1.7.3}-\eqref{eq1.7.6} approaches $(\phi_1,\phi_2,\phi_3,\phi_4)=(\phi_1^{\star},0,0,0)$ as $z\rightarrow-\infty$. Thus we necessarily have $\phi_1^{\star}=\frac{1}{\kappa}$.
	
\end{remark}

\subsection{Nonlinear terms in the model case}\lb{ssec2}

In this subsection we study the nonlinear terms defined in \eqref{eq2.1.6} and prove the nonlinearity is locally Lipschitz on the intersection space $\cE$.

Recall that we introduced the nonlinear term of system \eqref{eq2.1.7} as in formula \eqref{eq2.1.6}, that is
\begin{align}\label{nonlinear}
	H(\bv(t, \bx))&=f(\bu_-+\bv(t,\bx))-f(\bu_-)-\partial_{\bu}f(\bu_-)v(t,\bx)\no\\
	&=f\left(\begin{pmatrix}
		1/\kappa +v_1 \\ v_2
	\end{pmatrix} \right)-\begin{pmatrix}
		0 \\ 0
	\end{pmatrix}-\begin{pmatrix}
		0 & e^{-\kappa}\\
		0 &-\kappa e^{-\kappa}
	\end{pmatrix}\begin{pmatrix}
		v_1 \\v_2
	\end{pmatrix}\no\\
	&=\begin{pmatrix}
		v_2(e^{-\frac{1}{v_1+1/\kappa}}-e^{-\kappa}) \\ -\kappa v_2(e^{-\frac{1}{v_1+1/\kappa}}-e^{-\kappa})
	\end{pmatrix}.
\end{align}
To obtain the Lipschitz property of the nonliner term on the multidimensional space $H^k(\mathbb{R}^d)$ and $H^k_{\alpha}(\mathbb{R})\otimes H^k(\mathbb{R}^{d-1})$, we need the space $\mathcal{E}$ defined in equation \eqref{spE}.

It will be convenient to write $H(\bv)$ as follows:
\begin{equation}\label{eq2.3.1}
	H(\bv)=\begin{pmatrix}
		1 \\ -\kappa
	\end{pmatrix}\big( g(\frac{1}{\kappa}+v_1)-g(\frac{1}{\kappa}) \big)v_2,
\end{equation}
where $\bv=(v_1,v_2)$ and $g(\cdot)$ is defined as in \eqref{MPr}.

The proofs below will be based on the fact that the Sobolev embedding yields the inequality
\begin{equation}\label{eq2.3.3}
	||uv||_{H^k(\bbR^d)}\leq C||u||_{H^k(\bbR^d)}||v||_{H^k(\bbR^d)}
\end{equation}
for $2k>d$ (see \cite[Theorem 4.39]{AF}). We begin with some elementary facts.

\begin{lemma}\label{soblev}
	Assume that $k\geq[\frac{d+1}{2}]$, and consider $\mathcal{E}_0=H^k(\mathbb{R}^d)$. Then the following assertions hold.
	\begin{itemize}
		\item[(1)] If $u$, $v\in \mathcal{E}_0$, then $uv\in \mathcal{E}_0$, and there exists a constant $C>0$ such that $\|uv\|_{0}\leq C\|u\|_{0}\|v\|_{0}$.
		\item[(2)] If $u$, $v\in \mathcal{E}$, then $uv\in \mathcal{E}_{\alpha}$, and there exists a constant $C>0$ such that $\|uv\|_{\alpha}\leq C\|u\|_{0}\|v\|_{\alpha}$.
		\item[(3)] If $u$, $v\in \mathcal{E}$, then $uv\in \mathcal{E}$, and there exists a constant $C>0$ such that $\|uv\|_{\mathcal{E}}\leq C\|u\|_{\mathcal{E}}\|v\|_{\mathcal{E}}$.
	\end{itemize}
\end{lemma}
\begin{proof}
	Assertion (1) is in fact the Sobolev embedding inequality \eqref{eq2.3.3}.
	Assertion (2) can be proved by using \eqref{eq2.3.3} since
	\begin{equation*}
		\|uv\|_{\alpha}=\|\gamma_{\alpha}uv\|_{0}\leq C \|u\|_{0}\|\gamma_{\alpha}v\|_{0} = C\|u\|_{0}\|v\|_{\alpha}.
	\end{equation*}
	To show (3), let $u,v\in \mathcal{E}$. Then by (1), \begin{equation*}
		\|uv\|_{0}\leq C\|u\|_{0}\|v\|_{0}\leq C\|u\|_{\mathcal{E}}\|v\|_{\mathcal{E}},
	\end{equation*}
	and by (2), 
	\begin{equation*}
		\|uv\|_{\alpha}\leq C\|u\|_{0}\|v\|_{\alpha}\leq C\|u\|_{\mathcal{E}}\|v\|_{\mathcal{E}}.
	\end{equation*} 
	Therefore $uv\in\mathcal{E}$ and $\|uv\|_{\mathcal{E}}\leq C\|u\|_{\mathcal{E}}\|v\|_{\mathcal{E}}$.
\end{proof}

The nonlinearities of type \eqref{eq2.3.1} is a combination of the Nemytskij-type operator $v_1\mapsto g(1/\kappa+v_1)$ and multiplication operator by $v_2$. In what follows we will need to establish local Lipschitz properties of this and more general operators of the type $v\mapsto m(v(\cdot))v(\cdot)$ where $m(\cdot)$ is a given function and $v\in H^k(\bbR^d)$. The one-dimensional results of this type can be found in \cite[Proposition 7.2]{GLS2}. We presented an analogue of the proof of \cite[Proposition 7.2]{GLS2} in \cite[Appendix A]{GLY}, see Lemma \ref{glya} below.

\begin{lemma}\lb{glya}
	Assume $k\geq[\frac{d+1}{2}]$, and let $m:(q,u)\mapsto m(q,u)\in\mathbb{R}$ be a function from $C^{k+1}(\mathbb{R}^2)$. Consider the formula
	\begin{equation}\label{eq3.13}
		(q(\bx),u(\bx),v(\bx))\mapsto m(q(\bx),u(\bx))v(\bx),
	\end{equation}
	where $q(\cdot)$, $u(\cdot)$, $v(\cdot) :\mathbb{R}^{d}\mapsto \mathbb{R}$, and the variable $\bx=(x_1,...,x_d)\in\mathbb{R}^{d}$.
	\begin{itemize}
		\item[(1)] Formula \eqref{eq3.13} defines a mapping from $H^k(\mathbb{R}^d)\times\mathcal{E}_0^2$ to $\mathcal{E}_0$ that is locally Lipschitz on any set of the form $\{ (q,u,v):\|q\|_{0}+\|u\|_0+\|v\|_0\leq K \}$.
		\item[(2)] Formula \eqref{eq3.13} defines a mapping from $H^k(\mathbb{R}^d)\times\mathcal{E}^2$ to $\mathcal{E}$ that is locally Lipschitz on any set of the form $\{ (q,u,v):\|q\|_{0}+\|u\|_{\mathcal{E}}+\|v\|_{\mathcal{E}}\leq K \}$.
	\end{itemize}
\end{lemma}

By dropping $q$ from Lemma \ref{glya}, we record the following corollary that can be used to study the components of the map $H(\cdot)$ from \eqref{eq2.3.1}.

\begin{corollary}\label{cor2.3.3}
	Let $\cE_0=H^k(\bbR^d)$, and $\cE_{\alpha}$ and $\cE$ be defined accordingly. If $k\geq[\frac{d+1}{2}]$ and $m(\cdot)\in C^{\infty}(\mathbb{R})$, then the formula
	$$v(\bx)\mapsto m(v(\bx))v(\bx),\quad \bx\in\bbR^d,$$
	defines mappings from $\mathcal{E}_0$ to $\mathcal{E}_0$, and from $\mathcal{E}$ to $\mathcal{E}$. The first is locally Lipschitz on any set of the form $\{ v:\|v\|_0\leq K \}$; the second is locally Lipschitz on any set of the form $\{ v:\|v\|_{\mathcal{E}}\leq K \}$.
\end{corollary}

\begin{proposition}\label{lip2}
	Let $\cE_0=H^k(\bbR^d)$, and $\cE_{\alpha}$ and $\cE$ be defined accordingly, let $k\geq[\frac{d+1}{2}]$ and $\bv=\begin{pmatrix}
		v_1 \\ v_2
	\end{pmatrix}$, and consider the formula $$H(\bv)=\begin{pmatrix}
		1 \\ -\kappa
	\end{pmatrix}v_2\big( g(v_1+1/\kappa)-g(1/\kappa) \big)$$
	as given in \eqref{nonlinear}.
	\begin{itemize}
		\item[(1)] $H(\cdot)$ defines a mapping from $\mathcal{E}_0^2$ to $\mathcal{E}_0^2$ that is locally Lipschitz on any set of the form $\{ v: ||v||_0\leq K \}$.
		\item[(2)] $H(\cdot)$ defines a mapping from $\mathcal{E}^2$ to $\mathcal{E}^2$ that is locally Lipschitz on any set of the form $\{ v: ||v||_{\mathcal{E}}\leq K \}$.
	\end{itemize}
\end{proposition}
\begin{proof}
	It can be shown that $g(\cdot)\in C^{\infty}(\mathbb{R})$ is a smooth bounded function. Let 
	\begin{equation*}
		m(v_1)=g(v_1+1/\kappa)-g(1/\kappa),
	\end{equation*}then $m(\cdot)$ is also a smooth and bounded function. By applying Corollary \ref{cor2.3.3} to the components of the vector-valued map $H$, we finish the prooof.
\end{proof}

\begin{proposition}\label{lip3}
	Let $\cE_0=H^k(\bbR^d)$, let $\cE_{\alpha}$ and $\cE$ be defined accordingly, let $k\geq[\frac{d+1}{2}]$ and $\bv=\begin{pmatrix}
		v_1 \\ v_2
	\end{pmatrix}$, and consider the formula $$H(\bv)=\begin{pmatrix}
		1 \\ -\kappa
	\end{pmatrix}v_2\big( g(v_1+1/\kappa)-g(1/\kappa) \big)$$
	given in \eqref{nonlinear}.
	\begin{itemize}
		\item[(1)] If $\bv\in \mathcal{E}^2$, then there exist a constant $C_K>0$ such that $$||H(\bv)||_{\alpha}\leq C_K||\bv||_0||\bv||_{\alpha}$$ on any set of the form $\{ v: ||\bv||_{\mathcal{E}}\leq K \}$.
		\item[(2)] If $\bv\in \mathcal{E}^2$, then there exist a constant $C_K>0$ such that $$||H(\bv)||_{\mathcal{E}}\leq C_K||\bv||_{\mathcal{E}}^2$$ on any set of the form $\{v: ||\bv||_{\mathcal{E}}\leq K \}$.
	\end{itemize} 
\end{proposition}
\begin{proof}
	Note that 
	\begin{equation*}
		g(v_1+1/\kappa)-g(1/\kappa)=\int_0^tg'(1/\kappa+tv_1)dt\,v_1.
	\end{equation*}
	Since $\int_0^tg'(1/\kappa+tv_1)dt$ is also a smooth function, by Corollary \ref{cor2.3.3}, we have
	\begin{equation*}
		\|g(v_1+1/\kappa)-g(1/\kappa)\|_0\leq C_K\|v_1\|_0.
	\end{equation*}
	Also note that $\|v_1\|_0\leq \|\bv\|_0$ and $\|v_2\|_{\alpha}\leq \|\bv\|_{\alpha}$ because $v_1$ and $v_2$ are components of the vector $\bv$. Then (1) holds since
	\begin{align*}
		\|H(\bv)\|_{\alpha}
		=\|\gamma_{\alpha}H(\bv)\|_0&=\left\| \gamma_{\alpha}\begin{pmatrix}
			1 \\ -\kappa
		\end{pmatrix}v_2\big( g(v_1+1/\kappa)-g(1/\kappa) \big) \right\|_0\\
		&\leq C\|\big( g(v_1+1/\kappa)-g(1/\kappa) \big)\|_0\|\gamma_{\alpha}v_2\|_0\\
		&\leq C_K\|v_1\|_0\|v_2\|_{\alpha}\leq C_K||\bv||_0||\bv||_{\alpha}.
	\end{align*}
	Similarly, by using the fact that $\|v_2\|_0\leq \|\bv\|_0$, we have
	\begin{align*}
		\|H(\bv)\|_{0}
		&=\left\| \begin{pmatrix}
			1 \\ -\kappa
		\end{pmatrix}v_2\big( g(v_1+1/\kappa)-g(1/\kappa) \big) \right\|_0\\
		&\leq C\|\big( g(v_1+1/\kappa)-g(1/\kappa) \big)\|_0\|v_2\|_0\\
		&\leq C_K\|v_1\|_0\|v_2\|_{0}\leq C_K||\bv||_0||\bv||_{0},
	\end{align*}
	thus,
	\begin{align*}
		||H(\bv)||_{\mathcal{E}}&=\max\{ ||H(\bv)||_0,||H(\bv)||_{\alpha} \}\\
		&\leq\max\{ C_K||\bv||_0||\bv||_0,C_K||\bv||_0||\bv||_{\alpha} \}\\
		&\leq C_K||\bv||_{\mathcal{E}}||\bv||_{\mathcal{E}},
	\end{align*}
	and (2) is proved. 
\end{proof}

\subsection{Stability of the end state of the planar front in the model case}\lb{ssec3}
In this subsection we prove the stability of the end state $\bu_-=(1/\kappa, 0)$ of \eqref{eq2.1.3}.
By Proposition \ref{pro2.2.3} and Proposition \ref{lip2}, we know that, given initial data $\bv^0\in \mathcal{E}^2$, the system \eqref{eq2.1.7} has a unique mild solution $\bv(t,\bv^0)$. The solution is defined for $t\in[0, t_{max}(v))$ where $0<t_{max}(v))\leq \infty$; see, e.g., \cite[Theorem 6.1.4]{pazy}. The set $\{ (t,\bv^0)\in \mathbb{R}_+\times \mathcal{E}^2: 0\leq t < t_{max}(\bv))\}$ is open in $\mathbb{R}_+\times \mathcal{E}^2$, and the map $(t,\bv^0)\mapsto \bv(t,\bv^0)$ from this set to $\mathcal{E}^2$ is continuous; see, e.g., \cite[Theorem 46.4]{Sell}. We summarize these facts as follows.
\begin{proposition}\label{prop2.4.1}
	Let $\cE_0=H^k(\bbR^d)$ with $k\geq [\frac{d+1}{2}]$. For each $\delta >0$, if $0<\gamma <\delta$, then there exists $T(\gamma,\delta)$ depending on $\gamma$ and $\delta$, with $0<T(\gamma,\delta)\leq\infty$, such that the following is true: if $\bv^{0}\in \mathcal{E}^2$ satisfies
	\begin{equation}\label{eq1.15}
		||\bv^{0}||_{\mathcal{E}}\leq\gamma
	\end{equation}
	and $0\leq t<T$, then the solution $\bv(t)\in \mathcal{E}^2$ of \eqref{eq2.1.7} is defined and satisfies 
	\begin{equation}\label{eq1.16}
		||\bv(t)||_{\mathcal{E}}\leq \delta.
	\end{equation}
\end{proposition}

We can then prove the following proposition which shows that $\bv(t,\bv^0)\in\mathcal{E}^2$ is exponentially decaying in the weighted norm given $\bv^0$ is small in $\mathcal{E}^2$. We first establish exponential decay of the solutions of \eqref{eq2.1.7} on $H_{\alpha}^{k}(\mathbb{R}^d)^{2}$.
\begin{proposition}\label{prop2.4.2}
	Let $\mathcal{E}_0=H^k(\mathbb{R}^d)$ with $k\geq [\frac{d+1}{2}]$. Choose $\nu>0$ as in Proposition \ref{pro2.2.3}. Then there exist $\delta_1>0$, and $K_1>0$ such that for every $\delta\in(0,\delta_1)$ and every $\gamma$ with $0<\gamma<\delta$, the following is true: Let $\bv^{0}\in \mathcal{E}^2$ satisfies \eqref{eq1.15} so that $\bv(t)$ satisfies \eqref{eq1.16} for $0\leq t<T(\delta,\gamma) $. Then
	\begin{equation}
		||\bv(t)||_{\alpha}\leq K_1e^{-\nu t}||\bv^{0}||_{\alpha} \mbox{ \, for\,  }  0\leq t<T(\delta,\gamma).
	\end{equation}
\end{proposition}
\begin{proof}
     Because $\mathbf{v}(t)$ is a mild solution of \eqref{eq2.1.7} on $\mathcal{E}^{2}$, it satisfies the integral equation
     \begin{equation}\lb{int1}
     	\mathbf{v}(t)=e^{t \mathcal{L}_{\varepsilon}} \mathbf{v}^{0}+\int_{0}^{t} e^{(t-s) \mathcal{L}_{\mathcal{E}}} N(\mathbf{v}(s)) \mathbf{v}(s) d s .
     \end{equation}
     Since $\mathbf{v}^{0} \in \mathcal{E}^{2}$ by assumption, it is clear that $N(\mathbf{v}) \mathbf{v}$ is in $H_{\alpha}^{k}(\mathbb{R}^d)^{2}$ by Proposition \ref{lip3}, so we have
     $$
     e^{t \mathcal{L}_{\mathcal{E}}} \mathbf{v}^{0}=e^{t \mathcal{L}_{\alpha}} \mathbf{v}^{0} \text { and } e^{(t-s) \mathcal{L}_{\mathcal{E}}} N(\mathbf{v}(s)) \mathbf{v}(s)=e^{(t-s) \mathcal{L}_{\alpha}} N(\mathbf{v}(s)) \mathbf{v}(s)
     $$
     Next, we replace $\mathcal{L}_{\mathcal{E}}$ by $\mathcal{L}_{\alpha}$ in \eqref{int1} and choose $\bar{\nu}>\nu>0$ such that
     $$
     \sup \left\{\operatorname{Re} \lambda: \lambda \in \operatorname{Sp}\left(\mathcal{L}_{\alpha}\right)\right\}<-\bar{\nu}:=-k \nu,
     $$
     for some $k=\bar{\nu} / \nu>1$ and close to 1. There exists $K_{1}>0$ such that $\left\|e^{t \mathcal{L}_{\alpha}}\right\|_{\cE^2\rightarrow \cE^2} \leqslant$ $K_{1} e^{-\bar{\nu} t}$ for all $t \geqslant 0$ by Proposition \ref{pro2.2.3}.
     
     Pick any $\delta^{\prime}>0$. For any $\gamma$ so that $0<\gamma<\delta^{\prime}$, if $\left\|\mathbf{v}^{0}\right\|_{\mathcal{E}}<\gamma$. Then $\|\mathbf{v}(s)\|_{\mathcal{E}}<\delta^{\prime}$ for all $s \in\left(0, T\left(\delta^{\prime}, \gamma\right)\right)$ by Proposition \ref{prop2.4.1}.
     
     With the aid of Propostion \ref{lip3}.(1), there exists a constant $C_{\delta^{\prime}}>0$ depending on $\delta^{\prime}$ such that for $\|\mathbf{v}(s)\|_{\mathcal{E}} \leqslant \delta^{\prime}$ when $s \in\left(0, T\left(\delta^{\prime}, \gamma\right)\right)$, it follows that
     $$
     \|\mathbf{v}(t)\|_{\alpha} \leqslant K_{1} e^{-\bar{\nu} t}\left\|\mathbf{v}^{0}\right\|_{\alpha}+\int_{0}^{t} K_{1} e^{-\bar{\nu}(t-s)} C_{\delta^{\prime}}\|\mathbf{v}(s)\|_{0}\|\mathbf{v}(s)\|_{\alpha} d s
     $$
     For each $\delta<\delta^{\prime}$, and $0<\gamma<\delta$, if $\left\|\mathbf{v}^{0}\right\|_{\mathcal{E}}<\gamma$, then $\|\mathbf{v}(s)\|_{\mathcal{E}}<\delta$ for all $s \in$ $(0, T(\delta, \gamma))$ by Proposition \ref{prop2.4.1} again, then
     $$
     \|\mathbf{v}(t)\|_{\alpha} \leqslant K_{1} e^{-\bar{\nu} t}\left\|\mathbf{v}^{0}\right\|_{\alpha}+K_{1} C_{\delta^{\prime}} \delta \int_{0}^{t} e^{-\bar{\nu}(t-s)}\|\mathbf{v}(s)\|_{\alpha} d s
     $$
     Applying Gronwall's inequality for the function $e^{\bar{\nu} t}|| \mathbf{v}(t) \|_{\alpha}$, we can obtain that the inequality
     $$
     e^{\bar{\nu} t}\|\mathbf{v}(t)\|_{\alpha} \leqslant K_{1}\left\|\mathbf{v}^{0}\right\|_{\alpha}+K_{1} C_{\delta^{\prime}} \delta \int_{0}^{t} e^{\bar{\nu} s}\|\mathbf{v}(s)\|_{\alpha} d s
     $$
     implies, by Gronwall's inequality again, that
     $$
     \|\mathbf{v}(t)\|_{\alpha} \leqslant K_{1}\left\|\mathbf{v}^{0}\right\|_{\alpha} e^{K_{1} C_{\delta^{\prime}} \delta t-\bar{\nu} t} .
     $$
     By choosing $\delta_{1}<\min \left\{\delta^{\prime},(k-1) \frac{\nu}{K_{1} C_{\delta^{\prime}}}\right\}$, we can conclude that (2.33) holds for any $\delta \in\left(0, \delta_{1}\right)$.
\end{proof}

On the unweighted space $\mathcal{E}_0^2=H^k(\bbR^d)^2$, we rewrite system \eqref{eq2.1.7} as follows, 
\begin{align}
	\lb{sep1}&v_{1t}=L^{(1)} v_1+e^{-\kappa} v_2+H_0(\bv),\\
	\lb{sep2}&v_{2t}=L^{(2)} v_2-\kappa H_0(\bv),
\end{align}
where $H_0(\bv)=v_2\big(g(v_1+1/\kappa)-g(v_1)\big)$. Using Proposition \ref{lip2} and Propostion \ref{lip3}, we conclude that $H_0(\cdot)$ defines a mapping from $\mathcal{E}_0^2$ to $\mathcal{E}_0$ that is locally Lipschitz on any set of the form $\{ \bv:\|\bv\|_0\leq K \}$ and $||H_0(\bv)||_{0} \leq C_{K}||\bv||_{0}^2.$
Therefore, we obtain the following estimate.

\begin{proposition}\label{prop2.4.3}
	Let $\mathcal{E}_0=H^k(\mathbb{R}^d)$ with $k\geq [\frac{d+1}{2}]$. Choose $\rho>0$ as in Lemma \ref{lem2.2.4}.(3b), and $\delta_1$ be given by Proposition \ref{prop2.4.2}. Assume that $\nu<\rho$ where $\nu$ are chosen as in Proposition \ref{pro2.2.3}. Then there exist $\delta_2\in (0,\delta_{1})$ and $C_1>0$ such that for every $\delta\in(0,\delta_2)$ and every $\gamma$ with $0<\gamma<\delta$, the following is true: If $0\leq t<T(\delta,\gamma)$, and $\bv^0\in\mathcal{E}^2$ satisfies  \eqref{eq1.15}, so that the solution $\bv(t)\in\mathcal{E}^2$ of \eqref{eq2.1.7} satisfies \eqref{eq1.16}, then the following estimates hold: 
	\begin{align}
		\lb{estv1}&||v_1(t)||_{0}\leq C_1||\bv^{0}||_{\mathcal{E}},\\
		\lb{estv2}&||v_2(t)||_{0}\leq C_1e^{-\rho t}||\bv^{0}||_{\mathcal{E}}.
	\end{align}
\end{proposition}
\begin{proof}
	We note that $\left(v_{1}, v_{2}\right)$ is the solution of \eqref{sep1}-\eqref{sep2} with initial values $\left(v_{1}^{0}, v_{2}^{0}\right)$ at $t=0$, that is $\left(v_{1}, v_{2}\right)(t)=\left(v_{1}, v_{2}\right)\left(t, v_{1}^{0}, v_{2}^{0}\right)$.
	With the help of Proposition \ref{lip3}.(2), we can find a constant $C_{\delta_{1}}>0$ so that
	\begin{equation}
		\left\|H_{1}\left(v_{1}, v_{2}\right)\right\|_{0} \leqslant C_{\delta_{1}}\left\|v_{1}\right\|_{0}\left\|v_{2}\right\|_{0},\lb{h1}
	\end{equation}
	and
	$$
	\left\|H_{2}\left(v_{1}, v_{2}\right)\right\|_{0}=\left\|-\kappa H_{1}(\mathbf{v})\right\|_{0} \leqslant C_{\delta_{1}}\left\|v_{1}\right\|_{0}\left\|v_{2}\right\|_{0} 
	$$
	when $\|\mathbf{v}\|_{0} \leqslant \delta_{1}$.
	The solution of \eqref{sep2} in $H^{k}(\mathbb{R}^d)$ can be written as
	$$
	v_{2}(t)=e^{t \mathcal{L}_{2}} v_{2}^{0}+\int_{0}^{t} e^{(t-s) \mathcal{L}_{2}} H_{2}\left(v_{1}(s), v_{2}(s)\right) d s .
	$$
	We then choose some $\bar{\rho}>\rho>0$ and $k=\bar{\rho} / \rho>1$ such that
	$$
	\sup \left\{\operatorname{Re} \lambda: \lambda \in \operatorname{Sp}\left(\mathcal{L}_{2}\right)\right\}<-\bar{\rho}:=-k \rho .
	$$
	By Lemma \ref{lem2.2.4}.(3), there exists $K_{2}>0$ such that $\left\|e^{t \mathcal{L}_{2}}\right\|_{H^{k}(\mathbb{R}^d)\rightarrow H^k(\bbR^d)} \leqslant K_{2} e^{-\bar{\rho} t}$. For each $\delta \in\left(0, \delta_{1}\right)$ and $\gamma \in(0, \delta)$, if $\left\|\mathbf{v}^{0}\right\|_{\mathcal{E}} \leqslant \gamma$ then
	$$
	\left\|v_{1}(s)\right\|_{0} \leqslant\left\|v_{1}(s)\right\|_{\mathcal{E}} \leqslant\|\mathbf{v}(s)\|_{\mathcal{E}} \leqslant \delta .
	$$
	By Proposition \ref{prop2.4.1}, we can obtain the following estimate for $v_{2}(t)$ by using \eqref{sep2}:
	$$
	\begin{aligned}
		\left\|v_{2}(t)\right\|_{0} & \leqslant K_{2} e^{-\bar{\rho} t}\left\|v_{2}^{0}\right\|_{0}+\int_{0}^{t} K_{2} e^{-\bar{\rho}(t-s)} C_{\delta_{1}}\left\|v_{1}(s)\right\|_{0}\left\|v_{2}(s)\right\|_{0} d s \\
		& \leqslant K_{2} e^{-\bar{\rho} t}\left\|v_{2}^{0}\right\|_{0}+\int_{0}^{t} K_{2} e^{-\bar{\rho}(t-s)} C_{\delta_{1}} \delta\left\|v_{2}(s)\right\|_{0} d s .
	\end{aligned}
	$$
	We then calculate
	$$
	\begin{aligned}
		e^{\bar{\rho} t}\left\|v_{2}(t)\right\|_{0} & \leqslant K_{2}\left\|v_{2}^{0}\right\|_{\mathcal{E}}+K_{2} C_{\delta_{1}} \delta \int_{0}^{t} e^{\bar{\rho} s}\left\|v_{2}(s)\right\|_{0} d s \\
		& \leqslant K_{2}\left\|\mathbf{v}^{0}\right\|_{\mathcal{E}}+K_{2} C_{\delta_{1}} \delta \int_{0}^{t} e^{\bar{\rho} s}\left\|v_{2}(s)\right\|_{0} d s .
	\end{aligned}
	$$
	By applying Gronwall's inequality to $e^{\bar{\rho} t}\left\|v_{2}(t)\right\|_{0}$, we infer that
	$$
	\left\|v_{2}(t)\right\|_{0} \leqslant K_{2}\left\|\mathbf{v}^{0}\right\|_{\mathcal{E}} e^{K_{2} C_{\delta_{1}} \delta t-\bar{\rho} t} .
	$$
	Let $\delta_{2}<\min \left(\delta_{1}, \frac{(k-1) \rho}{K_{2} C_{\delta_{1}}}\right)$, then for $\delta<\delta_{2}$ it follows that
	$$
	\left\|v_{2}(t)\right\|_{0} \leqslant K_{2}\left\|\mathbf{v}^{0}\right\|_{\mathcal{E}} e^{-\rho t} \text { for all } t \in[0, T(\delta, \gamma)).
	$$
	proving \eqref{estv2}. We proceed next to prove \eqref{estv1}. The solution of \eqref{sep1} in $H^{k}(\mathbb{R}^d)$ satisfies
	$$
	v_{1}(t)=e^{t \mathcal{L}_{1}} v_{1}^{0}+\int_{0}^{t} e^{(t-s) \mathcal{L}_{1}}\left(e^{-\kappa} v_{2}(s)+H_{1}\left(v_{1}(s), v_{2}(s)\right) d s .\right.
	$$
	First, because $\mathcal{L}_{1}$ generates a bounded semigroup by Lemma \ref{lem2.2.4}.(1), there exists a constant $K_{3}>0$, such that $\left\|e^{t \mathcal{L}_{1}}\right\|_{H^{k}(\mathbb{R}^d)\rightarrow H^k(\bbR^d)} \leqslant K_{3}$. By using \eqref{h1} and the fact that
	$$
	\left\|e^{-\kappa} v_{2}(s)\right\|_{0} \leqslant\left\|v_{2}(s)\right\|_{0}
	$$
	for $\kappa>0$ we infer that
	$$
	\left\|v_{1}(t)\right\|_{0} \leqslant K_{3}\left\|v_{1}^{0}\right\|_{0}+\int_{0}^{t}\left(K_{3} C_{\delta_{1}}\left\|v_{2}(s)\right\|_{0}\left\|v_{1}(s)\right\|_{0}+K_{3}\left\|v_{2}(s)\right\|_{0}\right) d s .
	$$
	Also, using the fact that $\left\|v_{1}(s)\right\|_{0} \leqslant\|\mathbf{v}(s)\|_{0} \leqslant\|\mathbf{v}(s)\|_{\mathcal{E}}<\delta<\delta_{2}$, we have, for a constant $C_{\delta_{1}, \delta_{2}}>0$ independent of $\delta$, that
	$$
	\begin{aligned}
		\left\|v_{1}(t)\right\|_{0} & \leqslant K_{3}\left\|v_{1}^{0}\right\|_{\mathcal{E}}+\int_{0}^{t} K_{3}\left(C_{\delta_{1}}\left\|v_{1}(s)\right\|_{0}+1\right)\left\|v_{2}(s)\right\|_{0} d s\\
		& \leqslant K_{3}\left\|v_{1}^{0}\right\|_{\mathcal{E}}+\int_{0}^{t} K_{3} C_{\delta_{1}, \delta_{2}}\left\|v_{2}(s)\right\|_{0} d s.
	\end{aligned}
	$$
	Then we use \eqref{estv2} to obtain
	$$
	\begin{aligned}
		\left\|v_{1}(t)\right\|_{0} & \leqslant K_{3}\left\|\mathbf{v}^{0}\right\|_{\mathcal{E}}+\int_{0}^{t} K_{2} K_{3} C_{\delta_{1}, \delta_{2}} e^{-\rho s}\left\|\mathbf{v}^{0}\right\|_{\mathcal{E}} d s \\
		& \leqslant K_{3}\left\|\mathbf{v}^{0}\right\|_{\mathcal{E}}+K_{2} K_{3} C_{\delta_{1}, \delta_{2}}\left\|\mathbf{v}^{0}\right\|_{\mathcal{E}} \int_{0}^{t} e^{-\rho s} d s \\
		& \leqslant C_{2}\left\|\mathbf{v}^{0}\right\|_{\mathcal{E}}
	\end{aligned}
	$$
	for some $C_{2}>0$. In conclusion, there exists a constant $C_{1}>0$ such that for $\delta \in\left(0, \delta_{2}\right)$ and $\gamma \in(0, \delta)$, the inequalities \eqref{estv1} and \eqref{estv2} hold when $t \in[0, T(\delta, \gamma))$.
\end{proof}

\subsection{Proof of Theorem \ref{th2.1}}\lb{ssec4}
In this subsection we will present the main proof of the stability of the end state $\mathbf{u}_{-}$ of \eqref{eq2.1.3} in $\|\cdot\|_{\mathcal{E}}$. The proof relies on the following bootstrap argument based on Propositions \ref{prop2.4.1}- \ref{prop2.4.3}. Indeed, these propositions yield the existence of constants $\delta_{0}>0$ and $C_{\delta_{0}}>0$ such that for every $\delta \in\left(0, \delta_{0}\right)$ and every $\gamma \in(0, \delta)$, there exists $T(\delta, \gamma)$, such that for every $t \in[0, T(\delta, \gamma))$ the inequalities
	\begin{equation}\lb{remeq1}
		\|\mathbf{v}(t)\|_{\mathcal{E}}<\delta \text { and }\|\mathbf{v}(t)\|_{\mathcal{E}} \leqslant C_{\delta_{0}}\left\|\mathbf{v}^{0}\right\|_{\mathcal{E}}
	\end{equation}
hold for the solution $\mathbf{v}(t)$ of \eqref{eq2.1.7} with the initial value $\mathbf{v}^{0} \in \mathcal{E}^2$ as long as $\left\|\mathbf{v}^{0}\right\|_{\mathcal{E}}<\gamma$. Let us show that for each $\delta \in\left(0, \delta_{0}\right)$, there is an $\eta$ such that if $\left\|\mathbf{v}^{0}\right\|_{\mathcal{E}}<\eta$ then $\|\mathbf{v}(t)\|_{\mathcal{E}}<\delta$ for all $t \geqslant 0$. Indeed, assuming $C_{\delta_{0}}>1$ with no loss of generality, set $\eta=\frac{\delta}{2 C_{\delta_{0}}}$ and assume $\left\|\mathbf{v}^{0}\right\|_{\mathcal{E}}<\eta$. Then $\|\mathbf{v}(T(\delta, \gamma))\|_{\mathcal{E}}<\delta / 2$ by \eqref{remeq1}, and thus the solution $\mathbf{v}$ with the initial value $\mathbf{v}(T(\delta, \gamma))$ ) satisfies \eqref{remeq1} again for $t \in[T(\eta, \gamma), 2 T(\eta, \gamma))$, again by Proposition \ref{prop2.4.1} and \ref{prop2.4.2}. So, these propositions can be applied for all $t \geqslant 0$, proving the stability. In addition, as long as these propositions are applicable, we obtain a more refined information about the behavior of the solution such as its boundedness in $\|\cdot\|_{0}$-norm and the exponential decay in $\|\cdot\|_{\alpha}$-norm, see items (3)-(5) of \ref{th2.1}. 

Given initial value $\mathbf{v}^0\in\mathcal{E}^2$, let $\mathbf{v}(t)=\mathbf{v}(t,\mathbf{v}^0)$ be the solution of \eqref{eq2.1.7} in $\mathcal{E}^2$, which we showed exists on at least a short time period. We will complete the proof of nonlinear stability of the end state $\mathbf{u}_-$ by obtaining control on solutions for all time $t$. For this we need to use the following general result, see \cite[Proposition 1.21]{Tao}.
\begin{lemma}\label{ABP}
	(Abstract bootstrap principle) \\
	Let I be a time interval and suppose that for each $T\in I$ we have two statements: a "hypothesis" $H(T)$ and a "conclusion" $C(T)$. Let us suppose that we can verify the following four assertions:
	\begin{itemize}
		\item [(a)](Hypothesis implies conclusion) If $H(T)$ is true for some time $T\in I$, then $C(T)$ is also true for that time $T$.\\
		\item[(b)] (Conclusion is stronger than hypothesis) If $C(T)$ is true for some time $T\in I$, then $H(T')$ is true for all $T'\in I$ in a neighborhood of $T$.\\
		\item[(c)] (Conclusion is closed) If $T_1, T_2,...$ is a sequence of times in $I$ which converges to another time $T\in I$, and $C(T_n)$ is true for all $T_n$, then $C(T)$ is true.\\
		\item[(d)](Base case) $H(T)$ is true for at least one time $T\in I$.\\
		Then $C(T)$ is true for all $T\in I$. 
	\end{itemize}
\end{lemma} 

We will then state the main result in Theorem \ref{th2.1}. The small constant $\delta_{0}$ in the proof can be taken as $\delta_{0}=\delta_{2}$, where $\delta_{2}$ is chosen as in Proposition \ref{prop2.4.3}.

\begin{theorem} \label{th2.1}
	Let $\cE_0=H^k(\bbR^d)$ with $k\geq\frac{d+1}{2}$ and consider the semilinear system \eqref{eq2.1.7}. There exist  constants $C>0$ ,$\nu>0$ and a small $\delta_0>0$ such that for each $0<\delta<\delta_0$, we can find $\eta$ satisfying $0<\eta<\delta$ such that  if $||\bv^0||_{\mathcal{E}}\leq \eta$, the following is true for the solution $\bv(t)$ of \eqref{eq2.1.7} for all $t> 0$:
	\begin{itemize}
		\item[(1)] $\bv(t)$ is defined in $\mathcal{E}^2$;
		\item[(2)] $||\bv(t)||_{\mathcal{E}}\leq \delta$;
		\item[(3)] $||\bv(t)||_{\alpha}\leq Ce^{-\nu t}||\bv^0||_{\alpha}$;
		\item[(4)] $||v_1(t)||_{0}\leq C||\bv^0||_{\mathcal{E}}$;
		\item[(5)] $||v_2(t)||_{0}\leq Ce^{-\nu t}||\bv^0||_{\mathcal{E}}.$
	\end{itemize}
\end{theorem}

\begin{proof}
	Let $I$ be the time interval $[0,\infty)$.
	
	Let $H(T)$ in Lemma \ref{ABP} be the following statement: For each $0<\delta<\delta_2$, where $\delta_2$ is chosen as in Proposition \ref{prop2.4.3}, there exists $0<\gamma<\delta$, such that if $||\mathbf{v}^0||_{\mathcal{E}}\leq\gamma$, then $\mathbf{v}(t)$ is defined and $||\mathbf{v}(t)||_{\mathcal{E}}\leq\delta$ on the time interval $[0, T)$ for some $T=T(\delta,\gamma)$ depending on $\delta$ and $\gamma$. Thus  property (d)  of the bootstrap principle is proven by Proposition \ref{prop2.4.1}.
	
	Let $C(T)$ be the following statement: There exists $T>0$, such that on the time interval $[0, T)$, properties (3)-(5) in Theorem \ref{th2.1} hold.
	
	Let $0<\gamma_1<\delta<\delta_{2}$. Let $\gamma=C^{-1}\gamma_1$ where $C$ is a constant satisfying $C>\max\{1,K_1, C_1\}$ with $K_1$ and $C_1$ given as in Propositions \ref{prop2.4.2} and \ref{prop2.4.3}.
	
	Let $\mathbf{v}^0\in  \mathcal{E}^2$ with $||\mathbf{v}^0||_{\mathcal{E}}\leq\gamma$ be the initial value of \eqref{eq2.1.7}. Now $\gamma\leq \gamma_1< \delta$. Choose $\nu$ as in Proposition \ref{pro2.2.3}, by Propositions \ref{prop2.4.2} and \ref{prop2.4.3}, the items (3),(4) and (5) hold for $0<t\leq T(\delta,\gamma)$. So property (a) of the bootstrap principle is proven.
	
	Property (c) will hold by the continuity of $\mathbf{v}(t)$.
	
	Now we need to prove property (b) of the bootstrap principle.
	Let $\mathbf{v}^0\in \mathcal{E}^2$ with $||\mathbf{v}^0||_{\mathcal{E}}\leq\gamma$, for any $t \in(0,T)$, the inequilities (3), (4) and (5) hold, so by continuity of $\mathbf{v}(t)$ we can conclude:
	\begin{equation}
		||\mathbf{v}(T, \mathbf{v}^0)||_{\mathcal{E}}\leq C||\mathbf{v}^0||_{\mathcal{E}}\leq C\gamma= \gamma_1.
	\end{equation}
	If we take $\mathbf{v}^1=\mathbf{v}(T, \mathbf{v}^0)$ as an initial value of system \eqref{eq2.1.7}, it satisfies  $||\mathbf{v}^1||_{\mathcal{E}}\leq\gamma_1<\delta$, by applying Proposition \ref{prop2.4.1} again, there exists $T(\delta,\gamma_1)>0$, such that for all $t\in (T, T+T(\delta,\gamma_1))$, we have
	\begin{equation}
		||\mathbf{v}(T+t,\mathbf{v}^0)||_{\mathcal{E}}=||\mathbf{v}(t,\mathbf{v}^1)||_{\mathcal{E}}\leq\delta.
	\end{equation}
	Then $H(T')$ is true for $T'=T+T(\delta,\gamma_1)$ and property (b) is proven.
	
	Thus, by the bootstrap principle, we finish the proof of Theorem \ref{th2.1}.
\end{proof}

\section{Stability of the end states for a general system}\lb{sec3}
In this section we will study an $\bx$-independent steady state solution $\bu_{-}$to \eqref{eq1.1.2} with $f\left(\bu_{-}\right)=0$, and its perturbation depending on the spatial variable $\bx \in \mathbb{R}^{d}$.

Without loss of generality we shall take $\bu_{-}=0$. Information about the stability of the zero solution is encoded in the spectrum of the operator obtained by linearizing \eqref{eq1.1.2} about zero,
\begin{equation}\lb{opL}
	\bu_{t}=D \Delta_{\bx} u+c \partial_{z} \bu+\partial_{\bu} f(0) \bu=: L \bu,
\end{equation}
where $\partial_{\bu}$ is the differential with respect to $\bu$.

Let $\mathcal{E}_{0}$ be the Sobolev spaces $H^{k}\left(\mathbb{R}^{d}\right)$, and define the weight function
$$
\gamma_{\alpha}\left(z, x_{2}, \ldots, x_{d}\right)=e^{\alpha z}
$$
and the spaces $\mathcal{E}_{\alpha}$ and $\mathcal{E}=\mathcal{E}_{0} \cap \mathcal{E}_{\alpha}$ as before. Analogously to the model problem that has been discussed in Section \ref{sec2}, we use $\mathcal{L}$ to denote the operator defined on $\mathcal{E}_{0}^2$ given by the map $u \rightarrow L u$, with the domain $u \in H^{k+2}\left(\mathbb{R}^{d}\right)^2$, and use $\mathcal{L}_{\alpha}$ to denote the operator defined on $\mathcal{E}_{\alpha}^2$ given by $u \rightarrow L u$, with the domain being the set of $u$ where $\gamma_{\alpha} u \in H^{k+2}\left(\mathbb{R}^{d}\right)^2$.
Throughout we impose the following assumptions on $f(\cdot)$ in \eqref{eq1.1.2}.

\begin{hypothesis}\lb{hy1}
	\begin{itemize}
		\item [(a)]In appropriate variables $\bu=\left(\bu_{1}, \bu_{2}\right), \bu_{1} \in \mathbb{R}^{n_{1}}, \bu_{2} \in \mathbb{R}^{n_{2}}, n_{1}+n_{2}=n$, we assume that for some constant $n_{1} \times n_{1}$ matrix $A_{1}$, one has
		$$
		f\left(\bu_{1}, 0\right)=\left(A_{1} \bu_{1}, 0\right)^{T} .
		$$
		\item [(b)]The function $f$ is $C^{k+3}$ from $\mathbb{R}^{n}$ to $\mathbb{R}^{n}$.
	\end{itemize}
\end{hypothesis}
If Hypothesis \ref{hy1} holds, then $f(\bu_{\pm})=0$ and 
$$
\begin{aligned}
	f\left(\bu_{1}, \bu_{2}\right) &=f\left(\bu_{1}, 0\right)+f\left(\bu_{1}, \bu_{2}\right)-f\left(\bu_{1}, 0\right) \\
	&=\left(\begin{array}{c}
		A_{1} \bu_{1} \\
		0
	\end{array}\right)+\int_{0}^{1} \partial_{\bu_{2}} f\left(\bu_{1}, t \bu_{2}\right) d t \bu_{2}\\
&=\left(\begin{array}{c}
	A_{1} \bu_{1}+\tilde{f}_{1}\left(\bu_{1}, \bu_{2}\right) \bu_{2} \\
	\tilde{f}_{2}\left(\bu_{1}, \bu_{2}\right) \bu_{2}
\end{array}\right),
\end{aligned}
$$
where $\tilde{f}_{1}$ and $\tilde{f}_{2}$ are some matrix-valued functions of size $n_{1} \times n_{2}$ and $n_{2} \times n_{2}$, respectively.
We write
$$
D=\left(\begin{array}{cc}
	D_{1} & 0 \\
	0 & D_{2}
\end{array}\right), \quad f(\bu)=\left(\begin{array}{l}
	f_{1}\left(\bu_{1}, \bu_{2}\right) \\
	f_{2}\left(\bu_{1}, \bu_{2}\right)
\end{array}\right),
$$
where each $D_{i}$ is a nonnegative diagonal matrix of size $n_{i} \times n_{i}$, and $f_{i}: \mathbb{R}^{n_{1}} \times \mathbb{R}^{n_{2}} \rightarrow \mathbb{R}^{n_{i}}$ for $i=1,2$. Equation \eqref{eq1.1.2} now reads
\begin{equation}
	\partial_{t} \bu_{1}=D_{1} \Delta_{x} \bu_{1}+c \partial_{z} \bu_{1}+f_{1}\left(\bu_{1}, \bu_{2}\right), \lb{u1}
\end{equation}
\begin{equation}
	\partial_{t} \bu_{2}=D_{2} \Delta_{x} \bu_{2}+c \partial_{z} \bu_{2}+f_{2}\left(\bu_{1}, \bu_{2}\right). \lb{u2}
\end{equation}
If we linearize \eqref{u2} at $(0,0)$, then the constant-coefficient linear equation depends only on $u_{2}$ since $f(0,0)=0$ by Hypothesis \ref{hy1}.(a) can be obtained:
\begin{equation}\lb{l2}
	\begin{aligned}
		\partial_{t} \bu_{2} &=D_{2} \Delta_{\bx} \bu_{2}+c \partial_{z} \bu_{2}+\partial_{\bu_{1}} f_{2}(0,0) \bu_{1}+\partial_{\bu_{2}} f_{2}(0,0) \bu_{2} \\
		&=D_{2} \Delta_{\bx} \bu_{2}+c \partial_{z} \bu_{2}+\partial_{\bu_{2}} f_{2}(0,0) \bu_{2}.
	\end{aligned}
\end{equation}
We denote by $L^{(2)} \bu_{2}$ the right-hand side of \eqref{l2} and let $\mathcal{L}^{(2)}$ be the operator defined on $H^{k}\left(\mathbb{R}^{d}\right)^{n_{2}}$ given by $\bu \rightarrow L^{(2)} \bu$, with the domain $\bu \in H^{k+2}\left(\mathbb{R}^{d}\right)^{n_{2}}$.

In addition, we linearize \eqref{u1} at $(0,0)$, and by Hypothesis \ref{hy1}.(a), the respective constant-coefficient linear equation reads:
\begin{equation}\lb{l1}
	\begin{aligned}
		\partial_{t} \bu_{1} &=D_{1} \Delta_{\bx} \bu_{1}+c \partial_{z} \bu_{1}+\partial_{\bu_{1}} f_{1}(0,0) \bu_{1}+\partial_{\bu_{2}} f_{1}(0,0) \bu_{2} \\
		&=D_{1} \Delta_{\bx} \bu_{1}+c \partial_{z} \bu_{1}+A_{1} \bu_{1}+\partial_{\bu_{2}} f_{1}(0,0) \bu_{2} .
	\end{aligned}
\end{equation}
We denote $L^{(1)} \bu_{1}=D_{1} \Delta_{\bx} \bu_{1}+c \partial_{z} \bu_{1}+A_{1} \bu_{1}$, thus $\partial_{t} \bu_{1}=L^{(1)} \bu_1+\partial_{\bu_{2}} f_{1}(0,0) \bu_{2}$. Let $\mathcal{L}^{(1)}$ be the operator defined on $H^{k}\left(\mathbb{R}^{d}\right)^{n_{1}}$ given by $\bu \rightarrow L^{(1)} \bu$, with the domain $\bu \in H^{k+2}\left(\mathbb{R}^{d}\right)^{n_{1}}$.

With some additional assumptions listed below, we will show that the perturbations of the left end state $\bu_{-}$ that are initially small in both the unweighted norm and weighted norm stay small in the unweighted norm and decay exponentially in the weighted norm. In addition, the $\bu_{2}$-component of the perturbation decays exponentially in the unweighted norm. We will now use the following hypotheses about the spectrum of $\mathcal{L}$. 

\begin{hypothesis}\lb{h2}
	In addition to Hypothesis \ref{hy1}, we assume that there exists a constant $\alpha>0$ such that $\sup \left\{\operatorname{Re} \lambda: \lambda \in \operatorname{Sp}\left(\mathcal{L}_{\alpha}\right)\right\}<0$ on $L_{\alpha}^{2}(\mathbb{R})^{n} \otimes L^{2}\left(\mathbb{R}^{d-1}\right)^{n}$.
\end{hypothesis}
As in Subsection \ref{ssec1}, let $y=\left(x_{2}, \ldots, x_{d}\right) \in \mathbb{R}^{d-1}$, so that $\bx=(z, y) \in \mathbb{R}^{d}$, and denote $L_{1, \alpha}=D \partial_{z z}+c \partial_{z}+\partial_{\bu} f(0) \bu$ and $\Delta_{y}=\partial_{x_{2}}^{2}+\cdots+\partial_{x_{d}}^{2}$, and next define the linear operators $\mathcal{L}_{1, \alpha}: H_{\alpha}^{k}(\mathbb{R})^{n} \rightarrow H_{\alpha}^{k}(\mathbb{R})^{n}$, where $\operatorname{dom}\left(\mathcal{L}_{1, \alpha}\right)=H_{\alpha}^{k+2}(\mathbb{R})^{n} \subset H_{\alpha}^{k}(\mathbb{R})^{n}$, and $\Delta_{y}: H^{k}\left(\mathbb{R}^{d-1}\right)^{n} \rightarrow H^{k}\left(\mathbb{R}^{d-1}\right)^{n}$ where $\operatorname{dom}\left(\Delta_{y}\right)=H^{k+2}\left(\mathbb{R}^{d-1}\right)^{n} \subset H^{k}\left(\mathbb{R}^{d-1}\right)$. Then the operator $\mathcal{L}_{\alpha}$ on the space $H_{\alpha}^{k}\left(\mathbb{R}^{d}\right)^{n}$ can be represent as
$$
\mathcal{L}_{\alpha}=\mathcal{L}_{1, \alpha} \otimes I_{H^{k}\left(\mathbb{R}^{d-1}\right)}+I_{H_{\alpha}^{k}(\mathbb{R})} \otimes \Delta_{y} .
$$
Hypothesis \ref{h2} holds if there exist a constant $\alpha>0$ such that $\sup \{\operatorname{Re} \lambda: \lambda \in$ $\left.\operatorname{Sp}\left(\mathcal{L}_{1, \alpha}\right)\right\}<0$ on $H_{\alpha}^{k}(\mathbb{R})^{n}$. Indeed, by using Remark \ref{spectp}, we have
$$
\operatorname{Sp}\left(\mathcal{L}_{1, \alpha} \otimes I_{H^{k}\left(\mathbb{R}^{d-1}\right)}+I_{H_{\alpha}^{k}(\mathbb{R})} \otimes \Delta_{y}\right)=\operatorname{Sp}\left(\mathcal{L}_{1, \alpha}\right)+\operatorname{Sp}\left(\Delta_{y}\right)
$$
Note that the spectrum of $\mathcal{L}_{1, \alpha}$ on $L_{\alpha}^{2}\left(\mathbb{R}^{d}\right)$ and $H_{\alpha}^{k}\left(\mathbb{R}^{d}\right)$ are equal similarly to Lemma \ref{lem3sa}, and thus we can see that if $\sup \left\{\operatorname{Re} \lambda: \lambda \in \operatorname{Sp}\left(\mathcal{L}_{1, \alpha}\right)\right\}<0$ on $L_{\alpha}^{2}(\mathbb{R})^{n}$, then Hypothesis \ref{h2} will be satisfied for any $H_{\alpha}^{k}\left(\mathbb{R}^{d}\right)^n$.

\begin{hypothesis}\lb{h3}
	In addition to Hypothesis \ref{h2}, we assume the following:
	\begin{itemize}
		\item [(1)]The operator $\mathcal{L}^{(1)}$ generates a bounded semigroup on the spaces $L^{2}\left(\mathbb{R}^{d}\right)^{n_{1}}$ and $H^{k}\left(\mathbb{R}^{d}\right)^{n_{1}} .$
		\item[(2)] The operator $\mathcal{L}^{(2)}$ satisfies $\sup \left\{\operatorname{Re} \lambda: \lambda \in \operatorname{Sp}\left(\mathcal{L}^{(2)}\right)\right\}<0$ on $L^{2}\left(\mathbb{R}^{d}\right)^{n_{2}}$ and $H^{k}\left(\mathbb{R}^{d}\right)^{n_{2}} .$
	\end{itemize}
\end{hypothesis}

Note that we have used the following lemma in stating these hypotheses, which is an analog of Lemma \ref{lem2.2.1} in Subsection \ref{ssec1}.
\begin{lemma}\lb{lem3sa}
	The linear operator $\mathcal{L}$ associated with $L$ in \eqref{opL} have the same spectrum on $L^{2}\left(\mathbb{R}^{d}\right)^{n}$ and $H^{k}\left(\mathbb{R}^{d}\right)^{n}$, the linear operators $\mathcal{L}^{(i)}$ associated with $L^{(i)}$ in \eqref{l1} and \eqref{l2} have the same spectra on $L^{2}\left(\mathbb{R}^{d}\right)^{n_{i}}$ and $H^{k}\left(\mathbb{R}^{d}\right)^{n_{i}}$ for $i=1,2$; similarly, the linear operator $\mathcal{L}_{\alpha}$ has the same spectrum on both $L_{\alpha}^{2}(\mathbb{R})^{n} \otimes L^{2}\left(\mathbb{R}^{d-1}\right)^{n}$ and $H_{\alpha}^{k}(\mathbb{R})^{n} \otimes H^{k}\left(\mathbb{R}^{d-1}\right)^{n}$.
	
\end{lemma}
\begin{proof}
	Because $\mathcal{L}$ is associated with the constant-coefficient differential expression $L$, we can use the same proof as in Lemma \ref{lem2.2.1}.
\end{proof}

We will now rewrite equation \eqref{eq1.1.2} for the perturbation $\bv(t,\bx)$ of the end state $\bu_{-}=0$, in the form amenable for the subsequence analysis. We seek a solution to \eqref{eq1.1.2} of the form $\bu(t, \bx)=\bu_{-}+\bv(t, \bx)$. With this notation, $\bv=\bv(t, \bx)$ satisfies
\begin{equation}\lb{sys1}
	\bv_{t}=D \Delta_{\bx} \bv+c \partial_{z} \bv+\partial_{\bu} f(0) \bv+f(\bv)-f(0)-\partial_{\bu} f(0) \bv .
\end{equation}
Note that
$$
f(\bv)-f(0)-\partial_{\bu} f(0) \bv=\int_{0}^{1}\left(\partial_{\bu} f(t \bv)-\partial_{\bu} f(0)\right) d t \bv .
$$
We define
\begin{equation}\lb{nl}
	N(\bu)=\int_{0}^{1}\left(\partial_{\bu} f(t \bv)-\partial_{\bu} f(0)\right) d t,
\end{equation}
as an $n \times n$ matrix-valued function of $\bv$. Note that $N(\bv) \bv \in \mathbb{R}^{n}$ for any $\bv \in \mathbb{R}^{n}$. Using \eqref{nl}, we rewrite \eqref{sys1} as
\begin{equation}\lb{sysv}
	\bv_{t}=L \bv+N(\bv) \bv.
\end{equation}
This is the semilinear equation for the perturbation that we will study. Throughout the rest of this section, we will always assume $k \geqslant\left[\frac{d+1}{2}\right]$ for $\mathcal{E}_{0}=$ $H^{k}\left(\mathbb{R}^{d}\right)$ and $\mathcal{E}_{\alpha}=H_{\alpha}^{k}(\mathbb{R}) \otimes H^{k}\left(\mathbb{R}^{d-1}\right)$.

\begin{proposition}\lb{prop31}
Assume that Hypotheses \ref{hy1}-\ref{h3} hold. Then the following is true:
(1) There exists $\alpha>0$ such that on the weighted space $\mathcal{E}_{\alpha}^{n}$, the spectrum of $\mathcal{L}_{\alpha}$ will be bounded away from the imaginary axis: $\sup \left\{\operatorname{Re} \lambda: \lambda \in \operatorname{Sp}\left(\mathcal{L}_{\alpha}\right)\right\}<-\nu$ for some $\nu>0$. Also, there exists $K>0$ such that
$$
\left\|e^{t \mathcal{L}_{\alpha}}\right\|_{\mathcal{E}_{\alpha}^{n}\rightarrow \cE_{\alpha}^n} \leqslant K e^{-\nu t} \quad \text { for all } \quad t \geqslant 0 .
$$
(2) On the unweighted space $\mathcal{E}_{0}^{n_{2}}$, we have $\sup \left\{\operatorname{Re} \lambda: \lambda \in \operatorname{Sp}\left(\mathcal{L}^{(2)}\right)\right\}<-\rho$ for some $\rho>0$, and there exists $K>0$ such that $\left\|e^{t \mathcal{L}^{(2)}}\right\|_{\mathcal{E}_{0}^{n_{2}} \rightarrow \mathcal{E}_{0}^{n_{2}}} \leqslant K e^{-\rho t}$ for all $t \geqslant 0$.
\end{proposition}
\begin{proof}
	Statement (1) holds by Hypothesis \ref{h2} and Lemma \ref{lem3sa}. Statement (2) follows from Hypothesis \ref{h3} and Lemma \ref{lem3sa}.
\end{proof}

Proposition \ref{prop31} above gives the spectral stability of the linear operator in the semilinear system \eqref{sysv}. We next estimate the Locally Lipschitz property for the nonlinear terms $N(\bv)\bv$ as in \eqref{sysv} in the weighted and unweighted norms.
\begin{proposition}
	Assume $k \geqslant \frac{d+1}{2}$ and let $\mathcal{E}_{0}=H^{k}\left(\mathbb{R}^{d}\right)$. Given $f \in C^{k+3}\left(\mathbb{R}^{n} ; \mathbb{R}^{n}\right)$, consider the nonlinearity $N(\bv)$ defined in \eqref{nl}. Then we have:
	\begin{itemize}
		\item [(1)] If $\bv \in \mathcal{E}^{n}$, then $N(\bv) \bv \in \mathcal{E}_{\alpha}^{n}$, and on any bounded neighborhood of the form $\left\{\bv:\|\bv\|_{\mathcal{E}} \leqslant K\right\}$ there is a constant $C_{K}>0$ such that $\|N(\bv) \bv\|_{\alpha} \leqslant C_{K}\|\bv\|_{0}\|\bv\|_{\alpha}$.
		\item[(2)] If $\bv \in \mathcal{E}^{n}$, then $N(\bv) \bv \in \mathcal{E}_{0}^{n}$, and on any bounded neighborhood of the form $\left\{\bv:\|\bv\|_{0} \leqslant K\right\}$ there is a constant $C_{K}>0$ such that $\|N(\bv) \bv\|_{0} \leqslant C_{K}\|\bv\|_{0}^{2}$.
		\item[(3)] The formula $\bv \mapsto N(\bv) \bv$ defines a mapping from $\mathcal{E}^{n}$ to $\mathcal{E}^{n}$ that is locally Lipschitz on any bounded neighborhood of the form $\left\{\bv:\|\bv\|_{\mathcal{E}} \leqslant K\right\}$ in $\mathcal{E}^{n}$.
	\end{itemize}
\end{proposition}

Proof. We will refer to Lemma \ref{glya} for the components of $N(\bv) \bv$ by dropping $q$ from the lemma.
Note that
$$
N(\bv)=\int_{0}^{1}\left(\partial_{u} f(t \bv)-\partial_{u} f(0)\right) d t=\int_{0}^{1}\left(\int_{0}^{1} \partial_{u^{2}} f(s t \bv) d s\right) t v d t .
$$
By Lemma \ref{glya} applied to the components of the vector under the integral, the mapping $\bv \mapsto N(\bv)$ is locally Lipschitz on sets of the form $\left\{\bv:\|\bv\|_{0} \leqslant K\right\}$ and satisfies
$$
\|N(\bv)\|_{0} \leqslant C_{K}\|\bv\|_{0},
$$
Thus by Lemma 3.6(1) and (3.33), we conclude that the nonlinearity $N(\bv) \bv$ satisfies
$$
\|N(\bv) \bv\|_{0} \leqslant\|N(\bv)\|_{0}\|\bv\|_{0} \leqslant C_{K}\|\bv\|_{0}\|\bv\|_{0},
$$
while by Lemma 3.6(2) and (3.33), it satisfies
$$
\|N(\bv) \bv\|_{\alpha}=\left\|\gamma_{\alpha} N(\bv) \bv\right\|_{0} \leqslant\|N(\bv)\|_{0}\left\|\gamma_{\alpha} \bv\right\|_{0} \leqslant C_{K}\|\bv\|_{0}\|\bv\|_{\alpha},
$$
thus proving (1) and (2). Next, we use the definition of $\|\|_{\mathcal{E}}$ and infer
$$
\begin{aligned}
	\|N(\bv)\bv\|_{\mathcal{E}} &=\max \left\{\|N(\bv) \bv\|_{0},\|N(\bv) \bv\|_{\alpha}\right\} \\
	& \leqslant \max \left\{C_{K}\|\bv\|_{0}\|\bv\|_{0}, C_{K}\|\bv\|_{0}\|\bv\|_{\alpha}\right\} \\
	& \leqslant C_{K}\|\bv\|_{\mathcal{E}}\|\bv\|_{\mathcal{E}}.
\end{aligned}
$$

With the information that we have obtained now, the spectrum of the linear operator of system \eqref{sysv} is stable in the weighted space and the nonlinear terms of system \eqref{sysv} under the weighted norm satisfy certain locally Lipchitz conditions. Next we can proceed as in the proof of Propostions \ref{prop2.4.1}-\ref{prop2.4.3} as in Section \ref{sec2}, and finally use the similar Bootstrap arguments as in the proof of Theorem \ref{th2.1} to obtain the following stability result:

\begin{theorem}\lb{th2}
	Given initial value $\bv^{0} \in \mathcal{E}^{n}$, let $\bv(t)=v\left(t, \bv^{0}\right)$ be the solution of \eqref{sysv} in $\mathcal{E}^{n}$ with $\bv(0)=\bv^{0}$. Let $k \geqslant \frac{d+1}{2}, \mathcal{E}_{0}=H^{k}\left(\mathbb{R}^{d}\right)$ and $\mathcal{E}_{\alpha}=H_{\alpha}^{k}(\mathbb{R}) \otimes H^{k}\left(\mathbb{R}^{d-1}\right)$. Assume Hypotheses \ref{hy1}-\ref{h3}. Then there exist constants $C>0, \nu>0$ and a small $\delta_{0}>0$ such that for each $0<\delta<\delta_{0}$, we can find $\eta>0$ such that if $\left\|\bv^{0}\right\|_{\mathcal{E}} \leqslant \eta$, then the following is true for all $t>0$ :
	\begin{itemize}
		\item $\bv(t)$ is defined in $\mathcal{E}^{n}$;
		\item $\|\bv(t)\|_{\mathcal{E}} \leqslant \delta$;
		\item $\|\bv(t)\|_{\alpha} \leqslant C e^{-\nu t}\left\|\bv^{0}\right\|_{\alpha}$;
		\item $\left\|\bv_{1}(t)\right\|_{0} \leqslant C\left\|\bv^{0}\right\|_{\mathcal{E}}$;
		\item $\left\|\bv_{2}(t)\right\|_{0} \leqslant C e^{-\nu t}\left\|\bv^{0}\right\|_{\mathcal{E}}$.
	\end{itemize}
\end{theorem}
The proof is identical to the proof of Theorem \ref{th2.1}, so we will not restate it here.

\section{Conclusion and Future Work}
The study of traveling waves and their stability is a very meaningful topic, where stability of traveling wave solutions is one of the important properties in the qualitative analysis of traveling wave solutions of nonlinear equations. In this paper, we study a class of reaction diffusion equations that are characterized by a special product-triangle structure, in particular, the linear operator obtained after linearization of the system with respect to the traveling front has a triangular structure \eqref{trs}, and the nonlinear reaction terms have product structure. Some other systems that possess this type of structure include, for example, the exothermic-endothermic chemical reactions:
$$
\begin{aligned}
	&\partial_{t} y_{1}=\partial_{x x} y_{1}+y_{2} f_{2}\left(y_{1}\right)-\sigma y_{3} f_{3}\left(y_{1}\right) \\
	&\partial_{t} y_{2}=d_{2} \partial_{x x} y_{2}-y_{2} f_{2}\left(y_{1}\right) \\
	&\partial_{t} y_{3}=d_{3} \partial_{x x} y_{3}-\tau y_{3} f_{3}\left(y_{1}\right)
\end{aligned}
$$
Here $y_{1}$ is the temperature, $y_{2}$ is the quantity of an exothermic reactant, and $y_{3}$ is the quantity of an endothermic reactant. The parameters $\sigma$ and $\tau$ are positive, and there are positive constants $a_{i}$ and $b_{i}$ such that $f_{i}(u)=a_{i} e^{-\frac{b_{i}}{u}}$ for $u>0$ and $f_{i}(u)=0$ for $u \leqslant 0$. And the gasless combustion
$$
\partial_{t} u=\partial_{x x} u+v g(u), \quad \partial_{t} v=-\beta v g(u),
$$
where $g(u)=e^{-\frac{1}{u}}$ if $u>0$ and $g(u)=0$ if $u \leqslant 0$. In this system, $u$ is the temperature, $v$ is the concentration of unburned fuel, $g$ is the unit reaction rate, and $\beta>0$ is a constant parameter. 

For the reaction-diffusion system with this structure, we show that if the spectrum of the linear operator projected in one-dimensional space is only touching the imaginary axis, then we can use a weight function and weighted space to shift the spectrum of the linear operator to the left to obtain the spectral stability of the operator, on the other hand, we show that the nonlinear reaction term with the product form has the local lipchitz property in the weighted space we constructed. Combining these facts we can obtain the stability of the steady-state solution of the planar front.

However, there are still more problems in this subject that remain unsolved for the time being. For example, the linear operator obtained by linearizing the system about the planar front will have isolated singularities, and each of these isolated singularities extends an infinite semiline in the multidimensional space, as we discussed in the remark \ref{spectp}, which leads us to presuppose in \cite{GLY} that the diffusion coefficients of different variables of the system are identical, the reader can refer to \cite[Proposition 3.1]{GLY} if there is interest in the discussion.

Further studies of this subject may include, using the triangular structure, trying to control part of the variables followed by other parts; or, whether there are approaches other than spectral projection to study reaction diffusion systems in multidimensional space.


\bibliographystyle{amsplain}

\begin{thebibliography}{10}
	\bibitem{AF} R.\ A.\ Adams and J.\ F.\ Fournier, Sobolev spaces (2 ed.), Academic Press, New York, 2003.
	
	\bibitem{BGM} M.\ Bakhshi, A.\ Ghazaryan, V.\ Manukian, et al.  Traveling wave solutions in a model for social outbursts in a tension‐inhibitive regime, {\em Stud. Appl. Math.}, 2021.
	
	\bibitem{BFZ} B.\ Barker, H.\ Freistuhler and K.\ Zumbrun,  Convex entropy, Hopf bifurcation, and viscous and inviscid shock stability, {\em Arch. Ration. Mech. An.}, vol. 217, no. 1, pp. 309--372, 2015. 
	
	\bibitem{Brig} R.\ J.\ Briggs,  Electron-stream interaction with plasmas, Monograph MIT, Cambridge USA, vol. 29, 1964. 
	
	\bibitem{CL} D.\ Cramer, Y.\ Latushkin, Gearhart-Prüss theorem in stability for wave equations: a survey, Evolution equations. CRC Press, 2019: 105-119. 
	
	\bibitem{EngelNagel} K.-J. \ Engel and R. \ Nagel, One-parameter semigroups for linear evolution equations, Springer, New York, 1999.
	
	\bibitem{Fife2002} P.\ Fife, Pattern formation in gradient systems. In: {\em Handbook of dynamical systems}, Vol. 2, pp. 677--722, North-Holland, Amsterdam, 2002.
	
	\bibitem{G} A.\ Ghazaryan, Nonlinear stability of high Lewis number combustion fronts, {\em Indiana Univ. Math.\  J.,} {\bf 58} (2009), 181--212.
	
	\bibitem{GLS1} A.\ Ghazaryan, Y. Latushkin, S.\ Schecter and A. de Souza, Stability of gasless combustion fronts in one-dimensional
	solids, {\em Arch. Ration. \ Mech.\ Anal.}  {\bf 198} (2010) 981 -- 1030.
	
	\bibitem{GLS2} A.\ Ghazaryan, Y. Latushkin and S.\ Schecter, Stability of traveling waves for a class of reaction-diffusion systems that arise in chemical reaction models, {\em SIAM J.\ Math. Anal.} {\bf 42} (2010) 2434 -- 2472.
	
	\bibitem{GLS3} A.\ Ghazaryan, Y. Latushkin and S.\ Schecter. Stability of traveling waves for degenerate systems of reaction diffusion equations, {\em Indiana University Math.\ J.}  {\bf 60} (2011), 443 -- 472.
	
	\bibitem{GLSR} A.\ Ghazaryan, Y.\ Latushkin and S.\ Schecter, Stability of traveling waves in partly hyperbolic systems, {\em Math.\ Model.\ Nat.\ Phenom.} {\bf 8} (2013) 32 -- 48.
	
	\bibitem{GLY} A.\ Ghazaryan, Y.\ Latushkin and X.\ Yang, Stability of a planar front in a class of reaction-diffusion systems, SIAM J. Math. Anal., {\bf 50} (2018), 5569-5615.
	
	\bibitem{Henry1981}
	D. Henry. Geometric theory of semilinear parabolic equations. Lecture Notes in Mathematics vol. 840. Springer, New York, 1981.
	
	\bibitem{Jones} C.K.R.T. Jones, Geometric Singular Perturbation Theory, C.I.M.E. Lectures, Montecatini Terme, Lecture Notes in Math., Vol. 1609, Springer, Heidelberg (1995).
	
	\bibitem{Kapitula1} T.\ Kapitula, On the nonlinear stability of plane waves for the Ginzburg-Landau equation, {\em Comm. Pure Appl. Math.}{\bf 47}(1994), 831-841.
	
	\bibitem{kp} T.\ Kapitula, K.\ Promislow. An introduction to spectral and dynamical stability. Springer, New York, 2014.
	
	\bibitem{KV} B. Kazmierczak and V. Volpert. Travelling waves in partially degenerate reaction-diffusion systems. {\em Math. Model. Nat. Phenom}. {\bf 2}  (2007), 106--125.
	
	\bibitem{KSS}M. Krupa, B. Sandstede and P. Szmolyan,  Fast and slow waves in the FitzHugh-Nagumo equation, {\em J. Differential Equations}. {\bf 133}(1997), 49-97.
	
	\bibitem{LSY} Y.\ Latushkin, R.\ Schnaubelt and X.\ Yang,  Stable foliations near a traveling front for reaction diffusion systems, {\em Discrete Cont. Dyn.-B}, {\bf 22} (2017), 3145-3165.
	
	\bibitem{LMNT} V.\ Ledoux, S.\ Malham, J.\ Niesen and V.\ Th\"ummler, Computing stability of multidimensional traveling waves. {\em SIAM J. Appl. Dyn. Syst.} {\bf 8} (2009), 480 -- 507.
	
	\bibitem{liwu} Y. Li and Y. Wu. Stability of traveling front solutions with algebraic spatial decay for some auto-catalytic chemical reaction systems. {\em SIAM J. Math. Anal.}, {\bf 44} (2012),  1474--1521.
	
	\bibitem{pazy} P.\ Pazy. Semigroups of Linear Operators and Applications to Partial Differential Equations, Springer-Verlag, New York, 1983.
	
	\bibitem{pw94} R. L.  Pego,  M. I. Weinstein. Asymptotic stability of solitary waves. {\em Comm. Math. Phys.}, 164 (1994), 305--349. 
	
	\bibitem{RS1} M. Reed and B. Simon, Methods of modern mathematical physics, Vol. I, Analysis of operators, Academic Press, New York, 1978.
	
	\bibitem{ReedSimon4} M. Reed and B. Simon, Methods of modern mathematical physics, Vol. IV, Analysis of operators, Academic Press, New York, 1978.
	
	\bibitem{rott3} J. Rottmann-Matthes. Stability of parabolic-hyperbolic traveling waves. {\em Dynamics of Part. Diff. Eqns}.\ {\bf 9} (2012)  29--62.
	
	\bibitem{SS01} B.\ Sandstede and A.\ Scheel. Essential instabilities of fronts: bifurcation, and bifurcation failure. {\em Dyn. Syst.}, {\bf 16} (2001) 1, 1–28.
	
	\bibitem{Sa02} B.\ Sandstede.  Stability of travelling waves, pp. 983 -- 1055, In: 
	Handbook of dynamical systems, {V}ol. 2 (B. Fiedler, Ed.), North-Holland, Elsevier, Amsterdam, 2002.
	
	\bibitem{Sell} G. \ Sell and Y. \ You, Dynamics of Evolutionary Equations, {\em Appl. Math. Sci.}, 143, Springer-Verlag, New York, 2002.
	
	\bibitem{Tao} T. \ Tao, Nonlinear dispersive equations: local and global analysis, CBMS regional series in mathematics, AMS, Providence, 2006.
	
	\bibitem{TZKS} J.-C. Tsai, W. Zhang, V. Kirk and  J. Sneyd. Traveling waves in a simplified model of calcium dynamics.  {\em SIAM J. Appl. Dyn. Systems},  {\bf 11} (2012), 1149--1199.
	
	\bibitem{vv02} F.\ Varas,  J.\ Vega. Linear stability of a
	plane front in solid combustion at large heat of reaction. {\em SIAM
		J. Appl. Math.}, 62 (2002), 1810--1822.
	
	\bibitem{Xin2000} J.\ Xin, Front propagation in heterogeneous media. {\em SIAM Rev. }, {\bf 42} (2000), 161--230.
	
	
\end{thebibliography}


\end{document}